\documentclass[a4paper,twoside,10pt,reqno]{amsart}
\usepackage{amsmath,amsfonts,amssymb,amsthm}
\usepackage{mathrsfs}
\usepackage{graphicx}
\usepackage{color}
\usepackage[a4paper]{geometry}
\usepackage{cite}
\usepackage{booktabs}
\usepackage{tabularx}
\usepackage[colorlinks=true,pdfstartview={XYZ null null 1.00},pdfview=FitH,citecolor=blue,urlcolor=customgreen]{hyperref}
\usepackage{enumitem}
\usepackage{caption}
\newtheorem{theorem}{Theorem}[section]
\newtheorem{proposition}[theorem]{Proposition}
\newtheoremstyle{remarkstyle}
  {\medskipamount}   
  {\medskipamount}   
  {\normalfont}      
  {}                
  {\itshape}       
  {.}              
  {0.5em}       
{\thmname{#1}\thmnote{ #3}}
\theoremstyle{remarkstyle}
\newtheorem*{remark}{Remark}
\newtheorem*{remarks}{Remarks}
\numberwithin{equation}{section}
\newcommand\mytop[2]{\genfrac{}{}{0pt}{}{#1}{#2}}
\def\im{\mathrm{i}}
\def\d{\mathrm{d}}
\def\id{\,\mathrm{d}}
\def\e{\mathrm{e}}
\def\reg{\operatorname{reg}}

\def\Re{\operatorname{Re}}

\def\Im{\operatorname{Im}}
\definecolor{customgreen}{rgb}{0.0, 0.5, 0.0}

\author[G. Nemes]{Gerg\H{o} Nemes}
\address{School of Mathematics, Harbin Institute of Technology, Xidazhi Street, Harbin 150001,\newline Heilongjiang, People's Republic of China}

\email{g.nemes@hit.edu.cn}


\begin{document}

\title[Hyperasymptotics for linear difference equations]{Hyperasymptotics for linear difference equations\\ with an irregular singularity of rank one:\\ Polynomial coefficients}

\begin{abstract}
Hyperasymptotics is an analytical method that incorporates exponentially small contributions into asymptotic approximations, thereby expanding their domain of validity, improving accuracy, and providing deeper insight into the underlying singularity structures. It also allows for the computation of problem-specific invariants, such as Stokes multipliers, whose values are often assumed or remain unknown in other approaches. For differential equations, unlike standard asymptotic expansions, hyperasymptotic expansions determine solutions uniquely. In this paper, we extend the hyperasymptotic method to inverse factorial series solutions of certain higher-order linear difference equations and demonstrate that the resulting expansions also determine the solutions uniquely. We further indicate how the connection coefficients appearing in these expansions can be computed numerically using hyperasymptotic techniques. In addition, we give explicit remainder bounds for the inverse factorial series solutions. Our main tool is the Mellin--Borel transform. The expansions are expressed via universal hyperterminant functions, closely related to the hyperterminants familiar from integral and differential equation contexts. The results are illustrated by the Gauss hypergeometric function with a large third parameter and a third-order difference equation.
\end{abstract}

\maketitle

\section{Introduction}

In this paper, we investigate linear difference equations of order $n \ge 2$ of the form
\begin{equation}\label{diffeq}
w(z + n) + \sum_{k = 0}^{n - 1} f_k (z)w(z + k) = 0,
\end{equation}
where the coefficient functions $f_k(z)$ are polynomials of degree at most $n - k$, and $f_0(z)$ has degree exactly $n$. In particular, the point at infinity is an irregular singularity of rank $1$. It will be convenient to express these polynomials in the basis of falling factorials rather than the standard monomial basis:
\begin{equation}\label{fexpansions}
f_k(z) =\sum_{m = 0}^{n-k} f_{m,k} \frac{\Gamma (z + 1)}{\Gamma (z - n + k + m + 1)}  ,\quad k=0,1,\ldots,n-1,
\end{equation}
with $f_{0,0}\ne 0$.

Formal series solutions of \eqref{diffeq} in descending powers of $z$ are given by
\begin{equation}\label{formalseries}
    \Gamma (z)\lambda _j^{-z} z^{ \mu _j } \sum_{s = 0}^\infty  \frac{b_{s,j} }{z^s } ,\quad j=1,2,\ldots,n
\end{equation}
(see, for example, \cite{Immink1988}). The constants $\lambda_j$, $\mu_j$, and the coefficients $b_{s,j}$ are determined by substituting the formal series into the difference equation and matching coefficients, after freely assigning the initial term $b_{0,j}$.

Immink \cite{Immink1988} showed that each formal series in \eqref{formalseries} is Borel summable, and that its Borel transform (inverse Laplace transform) possesses infinitely many singularities in the complex Borel plane. These singularities give rise to infinitely many exponentially small contributions to the full asymptotic expansion of an exact solution, with \eqref{formalseries} representing the dominant level.

Here, we consider formal solutions in the form of inverse factorial series:
\begin{equation}\label{invfacseries}
\lambda_j^{-z} \sum_{s = 0}^\infty  a_{s,j} \Gamma \left(z + \mu _j  - s\right),\quad j=1,2,\ldots,n.
\end{equation}
Inverse factorial series play a significant role in the asymptotic analysis of integrals and differential equations, where they naturally arise as expansions for high-order coefficients and as tools for the numerical computation of connection coefficients, also called Stokes multipliers (see, e.g., \cite{Bennett2018,OldeDaalhuis1998b,Shelton2025}). More recently, they have also been shown to provide natural asymptotic representations for certain special functions of hypergeometric type \cite{Nemes2020,Nemes2025}. We remark that solutions of \eqref{diffeq} can also be expressed using (normal) factorial series \cite{Harris1963}, which converge, though slowly, in half-planes.

Substituting into the left-hand side of \eqref{diffeq} using \eqref{fexpansions} and \eqref{invfacseries}, and applying the known inverse factorial expansion for the ratio of gamma functions given in \cite[Eq.~\href{http://dlmf.nist.gov/5.11.E19}{(5.11.19)}]{DLMF}, we equate the coefficients of $\Gamma\left(z+n + \mu_j - s\right)$ to zero. This yields a hierarchy of equations determining the parameters $\lambda_j$, $\mu_j$, and the coefficients $a_{s,j}$. The characteristic equation
\begin{equation}\label{chareq}
\sum_{k = 0}^{n} f_{0,n-k} \lambda_j^k  = 0,
\end{equation}
with $f_{0,n}=1$, determines the values of $\lambda_j$. We impose the condition that
\[
\lambda_j \ne \lambda_\ell, \quad \ell\neq j.
\]
The corresponding constants $\mu_j$ are obtained from
\begin{equation}\label{mueq}
\mu_j  = 1-n+ \cfrac{\displaystyle\sum_{k = 1}^n f_{1,n-k} \lambda_j^{k} }{\displaystyle \sum_{k = 0}^{n} k f_{0,n-k} \lambda_j^k }=1-n- \cfrac{\displaystyle\sum_{k=1}^n f_{1,n-k} \lambda_j^k}{\displaystyle \prod_{\ell \neq j} \left( 1 - \frac{\lambda_j}{\lambda_\ell} \right)}.
\end{equation}
We also assume that the constants $\mu_j$ are non-integers. If this assumption fails, the independent variable can be shifted by replacing $z$ with $z + q$, choosing $q$ such that each $\mu_j + q$ becomes non-integer.  Finally, the coefficients $a_{s,j}$ satisfy the recurrence relation
\begin{gather}\label{arec}
\begin{split}
& s \prod_{\ell \ne j} \left( 1 - \frac{\lambda_j}{\lambda_\ell} \right) a_{s,j} \\ & = \sum_{m = 2}^{\min(n,s + 1)} a_{s - m + 1, j} 
\sum_{k = 1}^{n} \lambda_j^k 
\sum_{r = \max(0, m - k)}^{m} 
(-1)^r \frac{ \left( -n - \mu_j + s - r + 2 \right)_r  (m-k - r)_r }{r!} 
f_{m - r, n - k},
\end{split}
\end{gather}
valid for $s \ge 1$. The initial coefficient $a_{0,j}$ may be chosen freely. Here, $(\mu)_r = \Gamma(\mu + r)/\Gamma(\mu)$ denotes the Pochhammer symbol.

We note that an expansion of the form \eqref{formalseries} can always be transformed into one of the form \eqref{invfacseries}, and vice versa, with $\lambda_j$ and $\mu_j$ remaining unchanged. By this correspondence and the existing theory in the literature (see, for instance, \cite{Braaksma1979,Immink1988}), it is known that there exists a set of $n$ linearly independent solutions to the difference equation \eqref{diffeq} that are represented asymptotically by the formal series \eqref{invfacseries} as $\Re(z) \to +\infty$. In general, however, these solutions are not uniquely determined by their asymptotic expansions.

Olde Daalhuis \cite{OldeDaalhuis2004} studied inverse factorial series solutions to second-order linear difference equations of the form \eqref{diffeq}, under the milder assumption that $z^{-2} f_0(z)$ and $z^{-1} f_1(z)$ are analytic in a neighbourhood of infinity. He showed that the formal series \eqref{invfacseries} are Borel summable via Mellin-type integral transforms, and that their Borel transforms have only three singularities---located at $\lambda_1$, $\lambda_2$, and $0$---rather than infinitely many. The singularities at $\lambda_1$ and $\lambda_2$ are simple, corresponding to the two formal series solutions of the form \eqref{invfacseries} to equation \eqref{diffeq}. The singularity at the origin, however, is more intricate and generally more difficult to analyse. In the case where $\lambda_1$ is closer to $\lambda_2$ than to the origin, he showed that, after truncating the inverse factorial series \eqref{invfacseries} with $j = 1$ at its numerically least term, the remainder can be re-expanded into a new series whose coefficients are $a_{s,2}$. This results in an exponentially improved asymptotic expansion that uniquely determines an exact solution. A similar result holds for the second solution corresponding to $j = 2$. Furthermore, he derived asymptotic expansions for the coefficients $a_{s,j}$ as $s \to +\infty$, which can be used to compute the connection coefficients that appear in the exponentially improved asymptotic expansions. The assumption regarding the distance between $\lambda_1$ and $\lambda_2$ allows one to sidestep the complication of analysing the Borel transforms at the origin.

The analysis presented in \cite{OldeDaalhuis2004} can be extended to the higher-order case. The formal series \eqref{invfacseries} remain Borel summable via Mellin-type integral transforms, and their Borel transforms exhibit $n+1$ singularities located at $\lambda_1, \lambda_2, \ldots, \lambda_n$, and $0$. However, when $n\ge 3$, the exponentially improved asymptotic expansions typically no longer determine the solutions uniquely. In this paper, we address this issue by developing hyperasymptotic expansions for the solutions---refinements of the exponentially improved expansions---which not only ensure uniqueness but also provide increasingly accurate approximations. With each level of re-expansion, however, increasingly stricter conditions would be imposed on the relative positions of the singularities $\lambda_1, \lambda_2, \ldots, \lambda_n$ with respect to the origin. For this reason, we have restricted attention to the case where the coefficient functions $f_k(z)$ are polynomials. As we will see, this ensures that the Borel transforms of the inverse factorial series \eqref{invfacseries} remain analytic at the origin. We note that inverse factorial series solutions of second-order difference equations with polynomial coefficients were also studied by Olver \cite{Olver2000}.

Hyperasymptotic expansions were originally introduced by Berry and Howls \cite{BerryHowls1990} to refine WKB approximations for Schr\"odinger-type differential equations, incorporating exponentially small contributions. Their approach was entirely formal, based on the re-summation of the inverse factorial series for the WKB coefficients originally derived by Dingle \cite{Dingle1973}. Subsequent work extended hyperasymptotic analysis to single integrals with saddle points \cite{BerryHowls1991,Howls1992,Bennett2018} and to multidimensional integrals \cite{Howls1997,DelabaereHowls2002} within a rigorous mathematical framework. For multidimensional integrals, hyperasymptotics offered an effective solution to the global connection problem, where the geometric methods applicable to single integrals fail. Parallel approaches for differential equations, using the Cauchy--Heine method and Borel transforms, were developed by Olde Daalhuis and Olver \cite{OldeDaalhuisOlver1995,OldeDaalhuis1998b,OldeDaalhuis2005a,OldeDaalhuis2005b}. Again, hyperasymptotics enabled the numerical computation of Stokes multipliers (or connection coefficients)---often treated as unknowns in other methods---with arbitrary precision.

One notable difference between the hyperasymptotic expansions presented here and those studied previously is that, in earlier works, the region of validity expanded with successive re-expansions, whereas in our case it remains the same half-plane.

Hyperasymptotic expansions are expressed in terms of a universal family of special functions known as hyperterminants, which apply equally to problems arising from integrals and differential equations alike. Olde Daalhuis \cite{OldeDaalhuis1998a} developed efficient methods for computing these hyperterminants, which made hyperasymptotic expansions numerically feasible in practical applications.
The expansions presented in this paper are formulated in terms of hyperterminant functions closely related to these.

The remainder of the paper is organised as follows. In Section \ref{sec2}, we introduce the Borel transforms of the inverse factorial series \eqref{invfacseries} and their associated functions, and establish their fundamental properties. In Section \ref{sec3}, we develop hyperasymptotic refinements of the inverse factorial expansions using the Borel transforms and associated functions introduced in the previous section. Section \ref{sec4} discusses how the connection coefficients appearing in these hyperasymptotic expansions can be computed numerically via hyperasymptotic expansions for the late coefficients. In Section \ref{sec5}, we derive explicit error bounds for the inverse factorial expansions. Section \ref{sec6} illustrates the results with the Gauss hypergeometric function with a large third parameter and a third-order difference equation. The paper concludes with a discussion in Section \ref{sec7}.

\section{The Borel transforms and their associated functions}\label{sec2}

In this section, we introduce the Borel transforms of the inverse factorial series \eqref{invfacseries}, along with their associated functions, and establish their fundamental properties.

Before presenting the main results, we introduce some notation, which closely follows that of \cite{OldeDaalhuis1998b}. We define 
\[
\theta_{j,\ell} = \arg\left(\lambda_\ell - \lambda_j\right).
\]
We consider the complex $t$-plane, equipped with cuts from each $\lambda_j$ extending to infinity along the ray $\arg\left(t - \lambda_j\right) = \eta$ for a fixed real number $\eta$. This direction $\eta$ can be chosen arbitrarily, provided that no cut intersects any other singularity $\lambda_\ell$ with $\ell \ne j$; that is, we require
\[
 \eta \ne \theta_{j,\ell} \bmod{2\pi}, \quad 1 \leq j,\ell \leq n, \quad \ell \ne j.
\]
Any such choice of $\eta$ is called admissible.

For each $j = 1, 2, \dots, n$, we define the logarithm on the ray $\arg\left(t - \lambda_j\right) = \eta$ by
\[
\log\left(t - \lambda_j\right) = \log \left|t - \lambda_j\right| + \im(\eta-0).
\]
We denote the $t$-plane with these cuts and corresponding branches of the logarithm by $\mathcal{P}_\eta$. In this domain, each $\log\left(t - \lambda_j\right)$ is single-valued and consistent with the definition above along the ray $\arg\left(t - \lambda_j\right) = \eta$.

\begin{remark} In \cite{OldeDaalhuis1998b}, the author specifies the branch of the logarithm by requiring that $\log\left(t - \lambda_j\right) = \log \left|t - \lambda_j\right| + \im(\eta + 2\pi - 0)$. Since we adopt a different branch of the logarithm, the definitions of the loop contours and the connection coefficients will differ from the corresponding ones in the paper \cite{OldeDaalhuis1998b}.
\end{remark}

Given an admissible $\eta$, we define
\begin{gather}\label{etapm}
\begin{split}
\eta^- &= \inf\left\{ \widetilde{\eta} < \eta \mid \widetilde{\eta} \text{ is admissible for all } \widetilde{\eta} \in (\widetilde{\eta}, \eta] \right\}, \\
\eta^+ &= \sup\left\{ \widehat{\eta} > \eta \mid \widehat{\eta} \text{ is admissible for all } \widetilde{\eta} \in [\eta, \widehat{\eta}) \right\}, \\
\mathcal{I}_\eta &= \left(\eta^-, \eta^+\right).
\end{split}
\end{gather}
Note that the endpoints $\eta^-$ and $\eta^+$ themselves are not admissible.

Throughout, we also use the notation
\[
\int^{\left[\varphi\right]} = \int^{+\infty \e^{\im\varphi }}, \quad \varphi \in \mathbb{R}.
\]

We are now in a position to present the main results of this section, with proofs provided at the end. The following proposition introduces the Borel transforms of the inverse factorial series \eqref{invfacseries} and establishes their fundamental properties.

\begin{proposition}\label{prop1} Let $\eta \in \mathbb R$ be an admissible direction, and let $\mu_j$ $(1\le j\le n)$ be given by \eqref{mueq}. For each $1\le j\le n$, the function $y_j(t,\eta)$, defined by
\[
y_j (t,\eta) = \sum_{s = 0}^\infty  a_{s,j} \Gamma \left(\mu _j-s\right)\left( 1 - \frac{t}{\lambda _j } \right)^{s - \mu _j } ,\quad \left| t - \lambda _j \right| < \min_{\ell \ne j} \left|  \lambda _j  - \lambda _\ell \right|,
\]
is analytic in $\mathcal{P}_\eta$, satisfies the relation
\begin{equation}\label{connection}
y_j (t,\eta) = K_{\ell,j}\frac{2\pi \im}{ \e^{2\pi \im\mu _j } -1}y_\ell  (t,\eta) + \reg\left(t - \lambda _\ell \right),\quad \ell  \ne j,
\end{equation}
where the $K_{\ell,j }$ are constants, and admits endless analytic continuation along any path that does not contain any of the points $\lambda_1, \lambda_2, \ldots, \lambda_n$. Moreover, if $S$ is a sector in the complex $t$-plane of the form
\begin{equation}\label{Ssec}
S=\left\{ |t|>R, \;\alpha<\arg t<\beta\right\}
\end{equation}
with $0<\beta-\alpha<2\pi$ and $R>\max_\ell|\lambda_\ell|$, then for any fixed analytic continuation of $y_j(t,\eta)$ into the entire sector $S$, we have
\begin{equation}\label{Borellimit}
\lim_{t \to \infty} t^{-(a + \varepsilon)} y_j(t,\eta) = 0, \quad t \in S,
\end{equation}
for any $\varepsilon>0$, where $a$ denotes the real part of the root of the equation $f_0 (z) = 0$ with the largest real part.
\end{proposition}

In \eqref{connection}, $\reg\left(t - \lambda_\ell\right)$ denotes a function that is analytic in a neighbourhood of $t=\lambda_\ell$. The constant preceding $y_\ell(t,\eta)$ in \eqref{connection} is expressed in this seemingly complicated way to simplify the form of the hyperasymptotic expansions presented in the subsequent section. The constants $K_{\ell,j}$, referred to as connection coefficients or Stokes multipliers, can be computed using the numerical method outlined in Section \ref{sec4} and developed in \cite{OldeDaalhuis1998b}.

In the next proposition, we introduce functions associated with the Borel transforms $y_j (t,\eta)$. These are Laplace-type transforms of the functions $y_j (t,\eta)$ and will be used to derive the hyperasymptotic expansions.

\begin{proposition}\label{prop2} Let $\eta \in \mathbb R$ be an admissible direction. For each $1\le j\le n$, define
\begin{equation}\label{Yint} 
Y_j (u,\eta ) = \frac{1}{2\pi\im}\int_{\gamma _j (\eta )} \e^{u t} y_j (t,\eta)\id t,
\end{equation}
where $\gamma_j(\eta)$ is a contour in $\mathcal{P}_\eta$ that starts at infinity along the left-hand side of the branch cut $\arg\left(t-\lambda_j\right)=\eta$, encircles $\lambda_j$ in the positive (counter-clockwise) direction, and returns to infinity along the right-hand side of the same cut. Along the left-hand side of the branch cut, we have $\arg t = \eta-2\pi$, while along the right-hand side $\arg t = \eta$. The integral in \eqref{Yint} converges for
\[
u \in \mathscr{S}(\eta ) = \left\{ \tfrac{\pi }{2} - \eta  < \arg u < \tfrac{3\pi}{2} - \eta  \right\}
\]
and defines an analytic function that satisfies $Y_j(u, \eta) = Y_j(u, \widetilde{\eta})$ for all $\widetilde{\eta} \in \mathcal{I}_\eta$ and $u\in \mathscr{S}(\eta )\cap \mathscr{S}(\widetilde{\eta} )$. Moreover,
\begin{equation}\label{Ylarge}
Y_j(u,\eta)\sim \e^{\lambda _j u} u^{ - 1} \sum_{s = 0}^\infty  a_{s,j} \left( \lambda _j u\e^{ - \pi \im} \right)^{\mu _j  - s} 
\end{equation}
as $u \to \infty$ within $\mathscr{S}(\eta)$, and
\begin{equation}\label{Ysmall}
Y_j (u,\eta ) =\mathcal{O}\left(u^{-a-1-\varepsilon}\right)
\end{equation}
as $|u| \to 0$ in $\mathscr{S}(\eta)$, for any $\varepsilon > 0$, where $a$ is as defined in Proposition \ref{prop1}.
\end{proposition}

The following proposition presents exact analytic solutions of the difference equation \eqref{diffeq}, expressed either directly through Mellin-type integrals of the Borel transforms $y_j(t,\eta)$, or alternatively through Mellin transforms of the associated functions $Y_j(u, \eta)$. The large-$\Re(z)$ behaviour of these solutions are described by the inverse factorial series \eqref{invfacseries}.

\begin{proposition}\label{prop3} Let $\eta \in \mathbb{R}$ be an admissible direction such that $\left| \eta - \arg \lambda_j \right| < \pi$ and $- \frac{\pi}{2} + \eta^- < \arg \lambda _j  < \frac{\pi}{2} + \eta^+$, where $\eta_\pm$ are defined in \eqref{etapm}. For each $1 \le j \le n$, define
\begin{equation}\label{wdef}
w_j (z,\eta) =  - \frac{\Gamma (z + 1)}{2\pi \im}\int_{\gamma_j(\eta)} t^{ - z - 1} y_j (t,\eta)\id t, \quad \Re(z)>\max(-1,a),
\end{equation}
where $\gamma_j(\eta)$ is the contour specified in Proposition \ref{prop2}, subject to the additional condition that it does not encircle the origin. Then $w_j(z,\eta)$ is a solution of the difference equation \eqref{diffeq}, and satisfies $w_j(z,\eta) = w_j(z,\widetilde{\eta})$ for all $\widetilde{\eta} \in \mathcal{I}_\eta$. Moreover,
\begin{equation}\label{wint}
w_j(z, \eta) = \e^{-\pi \im z} \int_0^{[\pi-\omega]} u^z Y_j(u, \eta) \id u, \quad \Re(z)>a,
\end{equation}
where $\omega$ is an arbitrary angle satisfying $\left| \eta - \omega \right| < \frac{\pi}{2}$ and $\left| \arg \lambda_j - \omega \right| < \frac{\pi}{2}$. Finally,
\begin{equation}\label{winvfact}
w_j (z,\eta)\sim \lambda_j^{-z} \sum_{s = 0}^\infty  a_{s,j} \Gamma \left(z + \mu _j  - s\right),
\end{equation}
as $\Re(z) \to +\infty$, with $\Im(z)=\mathcal{O}\big(\sqrt{\Re(z)}\big)$.
\end{proposition}

\begin{remarks}
\noindent
\begin{enumerate}[label=(\roman*), leftmargin=*, labelindent=0pt,itemsep=0.3em]
    \item The restriction $\left| \eta - \arg \lambda_j \right| < \pi$ in Proposition \ref{prop3} ensures that the branch cut from $\lambda_j$ and the contour $\gamma_j(\eta)$ avoid the branch point of the integrand at $t=0$. The condition $- \frac{\pi}{2} + \eta^- < \arg \lambda_j < \frac{\pi}{2} + \eta^+$ is imposed to guarantee the validity of the inverse factorial expansion \eqref{winvfact}. A possible choice for $\omega$ in \eqref{wint} is $\omega = \frac{\eta + \arg \lambda_j}{2}$.
    \item In \cite{OldeDaalhuis2004}, the representation \eqref{wdef} was used to derive exponentially improved asymptotic expansions for solutions of second-order linear difference equations. In contrast, in this paper we use the alternative representation \eqref{wint} to develop hyperasymptotic expansions. This approach has the advantage that the hyperasymptotic expansions for $w_j(z, \eta)$ can be directly related to the hyperasymptotic expansions for Laplace-type integrals (such as \eqref{Yint}) established in \cite{OldeDaalhuis1998b}. Moreover, the polynomial growth of the Borel transforms $y_j(t,\eta)$ allows for certain simplifications compared to \cite{OldeDaalhuis1998b}.   
\end{enumerate}
\end{remarks}

\begin{proof}[Proof of Proposition \ref{prop1}]
We make a Mellin-type integral ``ansatz" to derive a linear differential equation for the Borel transforms $y_j(t)=y_j(t,\eta)$. Specifically, we seek functions $y_j(t)$ such that
\[
w_j (z)=w_j (z,\eta) =  - \frac{\Gamma (z + 1)}{2\pi \im}\int_{\gamma _j (\eta )} t^{ - z - 1} y_j (t)\id t 
\]
satisfies the difference equation \eqref{diffeq}. Note that, for each integer $k \ge 0$,
\[
w_j (z + k) =  - \frac{\Gamma (z+k + 1)}{2\pi \im}\int_{\gamma _j (\eta )} t^{ - z -k- 1} y_j (t)\id t =  - \frac{\Gamma (z + 1)}{2\pi \im}\int_{\gamma _j (\eta )} t^{ - z - 1} \frac{\d^k y_j (t)}{\d t^k }\id t,
\]
where the second equality follows by integrating by parts $k$ times. Consequently,
\begin{multline*}
f_k(z)w_j (z+k) = w_j (z+k)  \sum_{m = 0}^{n-k} f_{m,k} \frac{\Gamma (z + 1)}{\Gamma (z - n + k + m + 1)} \\ = - \frac{\Gamma (z + 1)}{2\pi \im}\int_{\gamma _j (\eta )} t^{ - z - 1} \left( \sum_{m = 0}^{n-k} f_{m,k} t^{n - k - m}\frac{\d^{n - m} y_j (t)}{\d t^{n - m} } \right)\id t.
\end{multline*}
Therefore, the left-hand side of \eqref{diffeq} can be expressed as
\[
w_j(z + n) + \sum_{k = 0}^{n - 1}f_k(z) w_j (z+k)  = - \frac{\Gamma (z + 1)}{2\pi \im}\int_{\gamma _j (\eta )} t^{ - z - 1} \left(  \sum_{m=0}^{n} g_m(t)\frac{\d^m y_j (t)}{\d t^m} \right)\id t,
\]
where
\[
g_m(t)= \sum_{k=0}^{m}  f_{n-m,m-k} t^k,
\]
with the convention $f_{0, n} = 1$. Thus, $w_j(z)$ satisfies the difference equation \eqref{diffeq} if and only if $y_j(t)$ satisfies the linear differential equation
\begin{equation}\label{ydiffeq}
\sum_{m=0}^n g_m(t)\frac{\d^m y_j (t)}{\d t^m} =0.
\end{equation}
From \eqref{chareq}, we obtain
\[
g_n(t)  = \prod_{m=1}^{n} \left( 1 - \frac{t}{\lambda_m} \right).
\]
Using \eqref{mueq} and the definition of $g_{n-1}(t)$, we can apply Lagrange interpolation to find
\[
g_{n-1}(t)  = \sum_{m=1}^{n} \left( \sum_{k=0}^{n-1} f_{1,k} \lambda_m^{-k} \right) \prod_{\substack{j=1 \\ j \ne m}}^n \cfrac{\displaystyle 1 - \frac{t}{\lambda_j}}{\displaystyle\frac{1}{\lambda_m} - \frac{1}{\lambda_j}} 
= \sum_{m=1}^{n} \frac{-\mu_m - n + 1}{\lambda_m} \prod_{\substack{j=1 \\ j \ne m}}^n \left( 1 - \frac{t}{\lambda_j} \right) .
\]
It follows that the differential equation \eqref{ydiffeq} is of Fuchsian type, with regular singular points at $\lambda_1,\lambda_2, \ldots, \lambda_n$ and at $t = \infty$. 

The roots of the indicial equation corresponding to the regular singularity $\lambda_j$, for $1 \le j \le n$, are easily seen to be $0, 1, \dots, n-2$, and $-\mu_j$. Consequently, the differential equation \eqref{ydiffeq} admits $n$ linearly independent solutions of the form
\begin{equation}\label{series}
\sum_{s = 0}^\infty  a_{s,j} \Gamma \left(\mu _j-s\right)\left( 1 - \frac{t}{\lambda _j } \right)^{s - \mu _j } , \quad 1 \le j \le n.
\end{equation}
Substituting these expansions into \eqref{ydiffeq} shows that the coefficients $a_{s,j}$ satisfy the recurrence relation \eqref{arec}.

On the domain $\mathcal{P}_\eta$, for each $1 \le j \le n$, the series in \eqref{series} converges in a neighbourhood of $\lambda_j$ with radius $\min_{\ell \ne j} \left| \lambda_j - \lambda_\ell \right|$, defining an analytic function there. We define the solution $y_j(t,\eta)$ of the differential equation \eqref{ydiffeq} as the single-valued analytic continuation of this function throughout $\mathcal{P}_\eta$. Considering $y_j (t,\eta)$ near $t=\lambda_\ell$ for $\ell \ne j$, we see that there exists a constant $K_{\ell,j}$ such that
\[
y_j (t,\eta) = K_{\ell,j }\frac{2\pi \im}{ \e^{2\pi \im\mu _j } -1}y_\ell  (t,\eta) + \reg\left(t - \lambda _\ell \right),
\]
where $\reg\left(t - \lambda _\ell \right)$ denotes a function analytic at $t = \lambda_\ell$. Finally, since \eqref{ydiffeq} is a linear Fuchsian differential equation, the solution $y_j(t,\eta)$ admits analytic continuation beyond $\mathcal{P}_\eta$ along any path in the complex plane that avoids the singular points $\lambda_1,\lambda_2, \dots, \lambda_n$. The roots of the indicial equation corresponding to the regular singularity at infinity coincide with the roots of the polynomial $f_0(z)$. It follows that, for any fixed analytic continuation of $y_j(t,\eta)$ into a sector $S$ of the form \eqref{Ssec}, we have
\[
y_j (t,\eta) = \mathcal{O}\big( t^{a} \log^{M-1} t \big),
\]
as $t \to \infty$ within $S$, where $a$ denotes the real part of the root of $f_0(z) = 0$ having maximal real part, and $M$ denotes its multiplicity. In particular, \eqref{Borellimit} follows.
\end{proof}

\begin{proof}[Proof of Proposition \ref{prop2}]
The convergence of the integral \eqref{Yint} for $u \in \mathscr{S}(\eta)$ follows directly from \eqref{Borellimit}. 

We assert that $Y_j(u, \eta)$ is independent of $\eta$ as long as $\eta$ varies within the interval $\left(\eta_-, \eta_+\right)$. To demonstrate this, let us take two values $\eta, \widetilde{\eta} \in \left(\eta_-, \eta_+\right)$ and a reference point $t_0$ (with respect to both $\eta$ and $\widetilde{\eta}$) lying on $\gamma_j(\eta)$. Now, if we rotate the contour $\gamma_j(\eta)$ so that it aligns with the direction $\widetilde{\eta}$---while still passing through $t_0$---it is easily seen that the value of the integral does not change (for $u \in \mathscr{S}(\eta) \cap \mathscr{S}(\widetilde{\eta})$), provided we define the integrand by fixing its value at $t = t_0$ and extending it analytically along the new path of integration. In doing so, we observe that the new integrand becomes $\e^{ut}y_j(t, \widetilde{\eta})$, and the rotated integration path can be taken as $\gamma_j(\widetilde{\eta})$. Hence,
\[
Y_j(u, \eta) = Y_j(u, \widetilde{\eta}) \; \text{ for } \; u \in \mathscr{S}(\eta) \cap \mathscr{S}(\widetilde{\eta}) \; \text{ whenever } \; \eta, \widetilde{\eta} \in \left(\eta_-, \eta_+\right).
\]

To establish the asymptotic expansion \eqref{Ylarge}, we proceed as follows. A simple change of variables allows us to rewrite \eqref{Yint} in the form
\[
Y_j(u, \eta) = \e^{\lambda_j u} \frac{\e^{\im( \eta - \pi)}}{2\pi \im} 
\int_{-\infty}^{(0+)} 
\e^{u \e^{\im( \eta - \pi)} t} 
y_j\big(t\e^{\im( \eta - \pi)} + \lambda_j, \eta\big) \id t,
\]
where the contour of integration encircles the negative real axis in the counter-clockwise direction. Along the lower side of the branch cut, we have $\arg t = -\pi$, while along the upper side $\arg t = \pi$. Applying Watson’s lemma for loop integrals \cite[p. 120]{Olver1997} then gives the asymptotic expansion \eqref{Ylarge}.

To analyse the behaviour of $Y_j(u, \eta)$ as $|u| \to 0$, we derive a linear differential equation satisfied by it. Integrating by parts $m$ times in \eqref{Yint} gives
\[
(-1)^m u^m Y_j(u, \eta) = \frac{1}{2\pi \im} \int_{\gamma_j(\eta)} \e^{ut} \frac{\d^m y_j(t, \eta)}{\d t^m} \id t.
\]
Differentiating $k$ times under the integral sign yields
\[
(-1)^m \sum_{r=0}^k \binom{k}{r} \frac{m!}{(m - k + r)!} u^{m - k + r} \frac{\d^r Y_j(u, \eta)}{\d u^r}
= \frac{1}{2\pi \im} \int_{\gamma_j(\eta)} \e^{ut} t^k \frac{\d^m y_j(t, \eta)}{\d t^m} \id t.
\]
It then follows directly that if $y_j(t, \eta)$ satisfies the linear differential equation \eqref{ydiffeq}, then $Y_j(u, \eta)$ satisfies
\[
\sum_{m=0}^n \sum_{k=0}^m (-1)^m f_{n-m,\,m-k} \sum_{r=0}^k \binom{k}{r} \frac{m!}{(m - k + r)!} u^{m - k + r} \frac{\d^r Y_j(u, \eta)}{\d u^r} = 0,
\]
that is, the linear differential equation
\[
\sum_{m=0}^n u^m \left( \sum_{k=0}^{n-m} \frac{u^k}{(m+k)!} \sum_{r=m+k}^n (-1)^r \binom{r-k}{m} r! f_{n-r, k} \right) \frac{\d^m Y_j(u, \eta)}{\d u^m} = 0.
\]
The origin is a regular singular point of this equation, and the corresponding indicial equation takes the form 
\[
\sum_{m=0}^n \left( \frac{1}{m!} \sum_{k=m}^n (-1)^k \binom{k}{m} k! f_{n-k,0} \right) \frac{\Gamma(z+1)}{\Gamma(z - m + 1)} = 0.
\]
The left-hand side of this equation can be simplified by applying the Chu--Vandermonde identity \cite[Eq.~\href{http://dlmf.nist.gov/15.4.E24}{(15.4.24)}]{DLMF} for the Gauss hypergeometric function ${}_2 F_1$:  
\begin{align*}
&\sum_{m=0}^n \left( \frac{1}{m!} \sum_{k=m}^n (-1)^k \binom{k}{m} k! f_{n-k,0} \right) \frac{\Gamma(z+1)}{\Gamma(z - m + 1)}\\ &\qquad= \sum_{k=0}^n (-1)^k k!   f_{n - k, 0} \sum_{m=0}^k \binom{k}{m} \frac{1}{m!}  \frac{\Gamma(z+1)}{\Gamma(z - m + 1)} =\sum_{k=0}^n (-1)^k k!   f_{n - k, 0} \times {}_2 F_1\!\left( \mytop{- k, - z}{1};1\right)
\\ &\qquad= \sum_{k=0}^n (-1)^k f_{n - k, 0}   \frac{\Gamma(z + k + 1)}{\Gamma(z + 1)}
= \sum_{k=0}^n f_{n - k, 0}  \frac{\Gamma(-z)}{\Gamma(-z - k)}
= f_0 (-z - 1).
\end{align*}
Therefore,
\[
Y_j (u,\eta) = \mathcal{O}\big( u^{-a-1} \log^{M-1} u \big),
\]
as $|u| \to 0$ within $\mathscr{S}(\eta)$, where $a$ denotes the real part of the root of $f_0(z) = 0$ having maximal real part, and $M$ denotes its multiplicity. In particular, \eqref{Ysmall} follows.
\end{proof}

\begin{proof}[Proof of Proposition \ref{prop3}] By Proposition \ref{prop1}, the integral \eqref{wdef} converges absolutely and uniformly for $\Re(z) > a$. The integrand is analytic in $\mathcal{P}_\eta$, except for a branch cut extending from $0$ to infinity in the direction $\arg\left(-\lambda_j\right)$. Reversing the steps that led to the differential equation \eqref{ydiffeq} for $y_j(t,\eta)$ in the proof of Proposition \ref{prop1}, it follows that $w_j(z,\eta)$ satisfies the difference equation \eqref{diffeq}.

An argument similar to that in the proof of Proposition \ref{prop2} shows that $w_j(z,\eta) = w_j(z,\widetilde{\eta})$ for all $\widetilde{\eta} \in \mathcal{I}_\eta$. Note that $\left| \eta - \arg \lambda_j \right| < \pi$ and $- \frac{\pi}{2} + \eta^- < \arg \lambda _j  < \frac{\pi}{2} + \eta^+$ imply $\left| \widetilde{\eta} - \arg \lambda_j \right| < \pi$ for all $\widetilde{\eta} \in \mathcal{I}_\eta$.

To prove the representation \eqref{wint}, we first observe that, by \cite[Eq.~\href{http://dlmf.nist.gov/5.9.E1}{(5.9.1)}]{DLMF},
\[
 \Gamma(z+1) t^{-z-1}= \int_0^{[-\omega]} \e^{-tv} v^z \id v =-\e^{-\pi \im z}\int_0^{[\pi-\omega]} \e^{tu} u^z \id u ,
\]
valid when $\left| \arg t - \omega \right| < \frac{\pi}{2}$ and $\Re(z)>-1$. Consequently, we have
\begin{align*}
w_j(z, \eta) 
& = -\frac{\Gamma(z+1)}{2\pi \im} \int_{\gamma_j(\eta)} t^{-z-1} y_j(t, \eta) \id t 
= -\frac{1}{2\pi \im} \int_{\gamma_j(\eta)} y_j(t, \eta)  \Gamma(z+1)  t^{-z-1} \id t 
\\ & = \e^{-\pi \im z} \frac{1}{2\pi \im} \int_{\gamma_j(\eta)} y_j(t, \eta) \left( \int_0^{[\pi - \omega]} \e^{t u} u^z\id u \right) \id t,
\end{align*}
under the conditions $\left| \eta - \omega \right| < \frac{\pi}{2}$ and $\Re(z) > \max(-1,a)$. Assuming also that $\left| \arg \lambda_j - \omega \right| < \frac{\pi}{2}$, and using the asymptotic properties of $Y_j(u, \eta)$, we may interchange the order of integration to obtain
\[
w_j(z, \eta) = \e^{-\pi \im z} \int_0^{[\pi - \omega]} u^z \left( \frac{1}{2 \pi \im} \int_{\gamma_j(\eta)} \e^{t u} y_j(t, \eta) \id t \right) \id u 
= \e^{-\pi \im z} \int_0^{[\pi - \omega]} u^z Y_j(u, \eta) \id u.
\]
Finally, the condition $\Re(z) > -1$ (when applicable) can be lifted by analytic continuation.

To derive the inverse factorial expansion \eqref{winvfact}, note that, since $\left| \eta - \arg \lambda_j \right| < \pi$ and $- \frac{\pi}{2} + \eta^- < \arg \lambda _j  < \frac{\pi}{2} + \eta^+$ there exists an angle $\widetilde{\eta} \in \mathcal{I}_\eta$ such that $\left| \widetilde{\eta} - \arg \lambda_j \right| < \frac{\pi}{2}$. Let $\widetilde{\eta}$ be such an angle. Then, by applying \eqref{Ylarge} and \eqref{Ysmall}, it follows that if $N > \max (0,\Re(\mu_j) + a )$, we have
\[
Y_j(u, \widetilde{\eta}) = \e^{\lambda_j u} u^{-1}  \sum_{s=0}^{N-1} a_{s,j} \left( \lambda_j u \e^{-\pi \im} \right)^{\mu_j - s} + \mathsf{R}_j (u, \widetilde{\eta};N) ,
\]
where the remainder term satisfies
\[
\left|\mathsf{R}_j (u, \widetilde{\eta};N) \right| \le A_{N,j} \left| \e^{\lambda_j u}u^{-1}\right|\left| \lambda_j u \right|^{\Re(\mu_j) - N},
\]
for all $u \in \mathscr{S}(\widetilde{\eta})$, with a suitable positive constant $A_{N,j}$. Next, let $\widetilde{\omega}$ be any angle satisfying $\left| \eta - \widetilde{\omega} \right| < \frac{\pi}{2}$ and $\left| \arg \lambda_j - \widetilde{\omega} \right| < \frac{\pi}{2}$. Then, we obtain
\begin{align*}
& \left| w_j(z, \eta) - \lambda_j^{-z} \sum_{s=0}^{N-1} a_{s,j} \Gamma\left( z + \mu_j - s \right) \right|=\left| w_j(z, \widetilde{\eta}) - \lambda_j^{-z} \sum_{s=0}^{N-1} a_{s,j} \Gamma\left( z + \mu_j - s \right) \right| \\
&\qquad= \left| \e^{-\pi \im z} \int_0^{[\pi - \widetilde{\omega}]} u^z \left( Y_j(u, \widetilde{\eta}) - \e^{\lambda_j u} u^{-1} \sum_{s=0}^{N-1} a_{s,j} \left( \lambda_j u \e^{-\pi \im} \right)^{\mu_j - s} \right) \id u \right| \\
&\qquad= \left| \e^{-\pi \im z} \int_0^{[\pi - \widetilde{\omega}]} u^z \mathsf{R}_j(u,\widetilde{\eta};N) \id u \right|  
= \left| \e^{-\pi \im z} \int_0^{[\pi - \arg \lambda_j]} u^z \mathsf{R}_j(u,\widetilde{\eta};N)  \id u \right| \\
&\qquad= \left| \lambda_j^{-z-1} \int_0^{+\infty} t^z \mathsf{R}_j\big( \lambda_j^{-1} \e^{\pi \im} t, \widetilde{\eta};N \big) \id t \right|  
\le \left| \lambda _j^{ - z}  \right|A_{N,j} \int_0^{ + \infty } t^{\Re(z + \mu _j ) - N - 1} \e^{ - t} \id t 
\\ & \qquad= \left| \lambda _j^{ - z} \right|A_{N,j}  \Gamma \left( \Re\left(z + \mu _j \right) - N\right),
\end{align*}
provided $\Re(z+\mu _j ) > N > \max (0,\Re(\mu_j) + a )$. In the fourth step, we rotated the path of integration from the direction $\pi - \widetilde{\omega}$ to $\pi - \arg \lambda_j$. This deformation is justified by Cauchy’s theorem, together with the fact that the integrand vanishes along the arc at infinity. Finally, if $\Im(z)=\mathcal{O}\big(\sqrt{\Re(z)}\big)$, then by the uniform asymptotics for the ratio of gamma functions \cite{Fields1972} we have
\[
 \Gamma\left( \Re(z + \mu_j) - N \right) = \mathcal{O}(1)  \Gamma\left( z + \mu_j - N \right),
\]
as $\Re(z)\to+\infty$.
\end{proof}

\section{Hyperasymptotic expansions}\label{sec3}

In this section, we develop hyperasymptotic refinements of the inverse factorial expansions \eqref{winvfact} by employing the Borel transforms and their associated functions studied in the previous section. To proceed, we first establish the necessary notation and definitions. Set
\[
\left.
\begin{aligned}
\lambda_{j,\ell} &= \lambda_j - \lambda_\ell, \\
\mu_{j,\ell} &= \mu_j - \mu_\ell,
\end{aligned}
\quad \right\}
\quad \ell \ne j \quad \text{and} \quad \widetilde{\mu} = \max\left\{\Re(\mu_1), \ldots, \Re( \mu_n)\right\}.
\]
Furthermore, for each non-negative integer 
$m$, define
\begin{gather}\label{alpha}
\begin{split}
\alpha_j^{(m)} & = \min\left\{|\lambda_j - \lambda_{j_1}| + |\lambda_{j_1} - \lambda_{j_2}| + \ldots + |\lambda_{j_m} - \lambda_{j_{m+1}}| :\right. \\ & \hspace{10em}\left.
j_1 \neq j,\, K_{j_1,j} \neq 0,\, j_\ell \neq j_{\ell-1},\, K_{j_\ell,j_{\ell-1}} \neq 0\right\},
\end{split}
\end{gather}
with $j_0=j$ in the case $m=0$. If we consider the directed graph $G=(V,E)$ with vertices $V = \left\{\lambda_1, \ldots, \lambda_n\right\}$ and edges
\[
E = \left\{(\lambda_p, \lambda_q)\mid 1 \le p, q \le n,\, p \neq q,\, K_{q,p} \neq 0\right\},
\]
then $\alpha_j^{(m)}$ represents the length of the shortest directed path consisting of $m+1$ edges starting at $\lambda_j$. We leave it to the reader to verify that the sequence $\alpha_j^{(m)}$ is strictly increasing and unbounded.

The hyperasymptotic expansions will be expressed in terms of hyperterminant functions, which are closely related to the hyperterminants commonly encountered in the study of integrals and differential equations. We begin by defining these latter functions. For a non-negative integer $\ell$, we set
\begin{equation}\label{hyperF}
\left.
\begin{aligned}
& F^{(1)} \!\left( z;\mytop{M_0}{\sigma_0} \right) = \int_0^{[\pi-\arg \sigma_0]} \frac{\e^{\sigma_0 t_0} t_0^{M_0-1}}{z - t_0}\id t_0,\\
&F^{(\ell + 1)}\!\left(z;\mytop{M_0,}{\sigma_0,}
\mytop{\ldots,}{\ldots,}
\mytop{M_\ell}{\sigma_\ell}\right) = \int_0^{[\pi - \arg \sigma _0]} \frac{\e^{\sigma _0 t_0 } t_0^{M_0  - 1} }{z - t_0 } \prod_{r=1}^\ell \int_0^{ [\pi-\arg \sigma _r]} \frac{\e^{\sigma _r t_r } t_r^{M_r  - 1} }{t_{r - 1}  - t_r }\id t_r \id t_0,
\end{aligned}
\; \right\}
\end{equation}
where $M_0,\ldots,M_\ell$ are arbitrary complex numbers with $\Re(M_r) > 1$ for $r = 0,\ldots,\ell$, and $\sigma_0,\ldots,\sigma_\ell$ are elements of the Riemann surface associated with the logarithm such that $\arg \sigma_r \neq \arg \sigma_{r+1} \bmod{2\pi}$ for $r = 0,\ldots,\ell-1$. These multiple integrals converge provided that $|\arg\left(\sigma_0 z\right)| < \pi$, and they can be analytically continued to other values of $\arg z$. The function $F^{(\ell+1)}$ is referred to as the $(\ell+1)^{\text{th}}$ hyperterminant. We note in particular that
\[
F^{(1)} \!\left( z;\mytop{M_0}{\sigma_0} \right) =\e^{\pi \im M_0} \e^{\sigma_0 z} z^{M_0 - 1} \Gamma \left(M_0\right)\Gamma \left(1 - M_0,\sigma_0 z\right),
\]
where $\Gamma(M,z)$ denotes the incomplete gamma function. For further properties of these hyperterminants, see \cite{OldeDaalhuis1998a}.

If for some index $r$ we have $\arg \sigma_r = \arg \sigma_{r+1} \bmod{2\pi}$, it becomes necessary to specify whether the $t_r$-contour passes to the left or right of the pole of $\left(t_{r-1} - t_r\right)^{-1}$. To resolve this ambiguity, we define
\[
F^{(\ell + 1)}\!\left(z;\mytop{M_0,}{\sigma_0,}
\mytop{\ldots,}{\ldots,}
\mytop{M_\ell}{\sigma_\ell}\right)  = \lim_{\varepsilon \to 0^+}
F^{(\ell + 1)}\!\left(z;\mytop{M_0,}{\sigma_0\e^{-\ell \varepsilon \im},}\mytop{M_1,}{\sigma_1\e^{-(\ell-1) \varepsilon \im},}
\mytop{\ldots,}{\ldots,}\mytop{M_{\ell-1},}{\sigma_{\ell-1}\e^{- \varepsilon \im},}
\mytop{M_\ell}{\sigma_\ell}\right).
\]
Thus, we require that the pole of $(t_{r-1}-t_r)^{-1}$ is on the left-hand side of the $t_r$-contour.

We now define the new hyperterminant functions. For any positive integer $\ell$, set
\begin{gather}\label{Hdef}
\begin{split}
H^{(\ell + 1)}\!\left(z;\mytop{M_0,}{\sigma_0,}
\mytop{\ldots,}{\ldots,}
\mytop{M_\ell}{\sigma_\ell}\right)
& =\frac{- \left(\e^{-\pi \im} \sigma_0\right)^z}{\Gamma\left(z + M_0 + 1\right)}\left(\e^{-\pi \im} \sum_{r = 0}^\ell \sigma_r \right)^{1-\ell+\sum_{r = 0}^\ell M_r} \\ &\quad \times
F^{(\ell + 1)}\!\left(0;\mytop{z + M_0 + 2,}{\sigma_0,}
\mytop{M_1,}{\sigma_1,}
\mytop{\ldots,}{\ldots,}
\mytop{M_\ell}{\sigma_\ell}\right),
\end{split}
\end{gather}
where $z$ and $M_0, \ldots, M_\ell$ are complex numbers satisfying $\Re(z + M_0 + 1) > 0$ and $\Re(M_r) > 1$ for $r = 1, \ldots, \ell$, and where $\sigma_0, \ldots, \sigma_\ell$ again lie on the Riemann surface of the logarithm. The new hyperterminant $H^{(\ell+1)}$ is essentially obtained by taking the Mellin transform of the standard hyperterminant $F^{(\ell)}$; this Mellin transform, in turn, corresponds to the standard hyperterminant of one level higher, $F^{(\ell+1)}$, evaluated at the special value $0$ of its argument. The normalisation is chosen so that the resulting hyperasymptotic expansions take a form as close as possible to those arising in the context of integrals and differential equations. For example, from \cite[Eq.~(3.2)]{OldeDaalhuis2009}, we have the identity
\[
H^{(2)}\!\left(z; \mytop{M_0,}{\sigma_0,}
\mytop{M_1}{\sigma_1} \right)
=
\frac{\left(\sigma_0 + \sigma_1\right)^{M_0 + M_1}}{\sigma_0^{M_0} \sigma_1^{M_1}}
\frac{
\Gamma\left(M_1\right)}{z + M_0 + M_1}
{}_2 F_1\!\left(
\mytop{1, M_1}{z + M_0 + M_1 + 1}
; 1 + \frac{\sigma_0}{\sigma_1}
\right),
\]
where ${}_2 F_1$ denotes the Gauss hypergeometric function. Appendix \ref{Appendix} presents an efficient method for computing these new hyperterminant functions.

As in the proof of Proposition~\ref{prop3}, we choose an angle $\widetilde{\eta} \in \mathcal{I}_\eta$ such that $\left| \widetilde{\eta} - \arg \lambda_j \right| < \frac{\pi}{2}$. Then,
\begin{equation}\label{wintegral}
w_j(z, \eta) = w_j(z, \widetilde{\eta}) 
= \e^{-\pi \im z} \int_0^{[\pi - \widetilde{\omega}]} u^z Y_j(u, \widetilde{\eta})\id u = \e^{-\pi \im z} \int_0^{[\pi - \arg \lambda_j]} u^z Y_j(u, \widetilde{\eta})\id u,
\end{equation}
provided that $\Re(z) > -a$. For a non-negative integer $N_j^{(0)}$, define the remainder $\mathsf{R}_j^{(0)}\big(u, \widetilde\eta; N_j^{(0)}\big)$ by
\begin{equation}\label{Yremainder}
Y_j(u, \widetilde\eta) = \e^{\lambda_j u} u^{-1} 
\sum_{s = 0}^{N_j^{(0)} - 1} a_{s,j} 
\left( \lambda_j u \e^{-\pi \im} \right)^{\mu_j - s} 
+ \mathsf{R}_j^{(0)}\big(u, \widetilde\eta;N_j^{(0)}\big) 
\end{equation}
(cf.~\eqref{Ylarge}). Substituting this into \eqref{wintegral} gives
\begin{equation}\label{wremainder}
w_j (z,\eta ) = \lambda _j^{ - z} \sum_{s = 0}^{N_j^{(0)}  - 1} a_{s,j} \Gamma \left( z + \mu _j  - s \right)  + R_j^{(0)} \big(z,\eta ;N_j^{(0)} \big),
\end{equation}
where
\begin{equation}\label{remainderformula}
R_j^{(0)} \big(z,\eta ;N_j^{(0)} \big) = \e^{ - \pi \im z} \int_0^{\left[ \pi  - \arg \lambda _j  \right]} u^z \mathsf{R}_j^{(0)} \big(u,\widetilde{\eta};N_j^{(0)} \big)\id u,
 \end{equation}
provided that $\Re (z) > \max (N_j^{(0)}  - \Re(\mu _j)  - 1,a)$.

The key idea in deriving the hyperasymptotic expansions is to apply the well-established theory of hyperasymptotic expansions for Laplace-type integrals---such as the one defining $Y_j(u, \widetilde\eta)$---as developed in~\cite{OldeDaalhuis1998b}. In particular, hyperasymptotic re-expansions of the remainder $\mathsf{R}_j^{(0)}$ will, via \eqref{remainderformula}, directly yield hyperasymptotic re-expansions of the remainder $R_j^{(0)}$ in the inverse factorial expansion of $w_j(z, \eta)$. To ensure the applicability of this procedure, we adopt a stronger assumption on $\eta$, namely that $\eta^- \le \arg \lambda_j \le \eta^+$, in place of the original condition $- \frac{\pi}{2} + \eta^- < \arg \lambda_j < \frac{\pi}{2} + \eta^+$. This guarantees that the required re-expansions of $\mathsf{R}_j^{(0)}$ are valid.

We follow the standard terminology for the expansions discussed in the subsections below. The original inverse factorial expansions are referred to as level-zero---hence the notation $(0)$ in $N_j^{(0)}$. Truncating the level-zero expansion at or near its numerically least term yields a superasymptotic approximation. A level-one expansion is constructed by truncating the level-zero expansion at or beyond its numerically least term and then re-expanding the resulting remainder using $H^{(2)}$ functions (Gauss hypergeometric functions). These level-one expansions are called exponentially improved expansions. Subsequent re-expansions of the remainders correspond to higher levels in the hyperasymptotic hierarchy.

\subsection{Superasymptotics} In this subsection, we examine the behaviour of the remainder $R_j^{(0)}$ in \eqref{wremainder} in the case where both $\Re(z)$ and $N_j^{(0)}$ are large and of comparable magnitude.

We begin by presenting an explicit expression for the remainder $\mathsf{R}_j^{(0)}$ in \eqref{Yremainder}, derived by following the procedure outlined in \cite[\S3]{OldeDaalhuis1998b}. The key difference in our setting is that the Borel transform grows at most polynomially, rather than exponentially, at infinity. As a result, the large circular contours used in \cite{OldeDaalhuis1998b} can be extended to infinity. We obtain
\begin{equation}\label{R0integral}
\mathsf{R}_j^{(0)}\big(u, \widetilde\eta;N_j^{(0)}\big)
= \sum_{j_1 \ne j} \frac{K_{j_1, j}}{2\pi \im} 
\int_{\lambda_j}^{[\widetilde\eta]} 
\int_{\gamma_{j_1}(\theta_{j,j_1})} 
\frac{\e^{u t_0}}{t_1 - t_0} 
\left( \frac{t_0 - \lambda_j}{t_1 - \lambda_j} \right)^{N_j^{(0)} - \mu_j} 
y_{j_1}(t_1, \widetilde\eta)  \id t_1  \id t_0,
\end{equation}
provided that $\arg u = \pi - \arg \lambda_j$ and $N_j^{(0)}  > \max(0,\Re(\mu _j ) - 1,\Re(\mu _j ) + a)$. Here, $\gamma_{j_1}(\theta_{j,j_1})$ is a contour that begins at infinity along the left-hand side of the branch cut $\arg\left(t_1 - \lambda_{j_1}\right) = \theta_{j,j_1}$, encircles $\lambda_{j_1}$ in the positive (counter-clockwise) direction, and returns to infinity along the left-hand side of the same cut. Along the left-hand side of the branch cut, we have $\arg t = \theta_{j,j_1}-2\pi$, while along the right-hand side $\arg t = \theta_{j,j_1}$. The phase $\theta_{j,j_1}$ of $\lambda_{j_1} - \lambda_j$ is chosen to lie in the interval $(\widetilde\eta - 2\pi, \widetilde\eta)$. If $|\lambda_j-\lambda_{j_2}| > |\lambda_j-\lambda_{j_1}|$ and $\theta_{j,j_1} = \theta_{j,j_2}$, then the contour $\gamma_{j_1}(\theta_{j,j_1})$ passes to the right of the contour $\gamma_{j_2}(\theta_{j,j_2})$.

By applying the method leading to \cite[Eq.~(3.10)]{OldeDaalhuis1998b}, we obtain
\begin{align*}
\mathsf{R}_j^{(0)}\big(u, \widetilde\eta;N_j^{(0)}\big)
& =  \e^{\lambda_j u}|u|^{\Re(\mu_j) - N_j^{(0)}}\Gamma\big(N_j^{(0)} - \Re\left(\mu_j\right) + 1\big) \\ &\quad\times \sum_{j_1 \ne j} K_{j_1, j}  
\left| \lambda_{j_1, j} \right|^{-N_j^{(0)}} 
\big(N_j^{(0)}\big)^{\Re(\mu_{j_1}) - 1}  \mathcal{O}(1) \\ &\quad
+ \e^{\lambda_j u - \alpha_j^{(0)} |u|} \sum_{j_1 \ne j} K_{j_1, j} 
\left( \frac{\alpha_j^{(0)}}{|\lambda_{j_1, j}|} \right)^{N_j^{(0)}} 
\big(N_j^{(0)}\big)^{\Re(\mu_{j_1}) - 1} \mathcal{O}(1),
\end{align*}
valid for large $u$ and $N_j^{(0)}$, with $\arg u = \pi - \arg \lambda_j$. The quantity $\alpha_j^{(0)}$ is defined in \eqref{alpha}. From \eqref{Ysmall}, we may infer that this estimate remains valid for all $u$ satisfying $\arg u = \pi - \arg \lambda_j$, after adjusting the implied constants if necessary. Substituting this result into \eqref{remainderformula} yields
\begin{gather}\label{R0estimate}
\begin{split}
 R_j^{(0)}\big(z,\eta ;N_j^{(0)} \big) &  = \lambda_j^{-z} \Gamma\big(N_j^{(0)} - \Re\left(\mu_j\right) + 1\big) \Gamma\big( \Re\left(z + \mu_j\right) - N_j^{(0)} + 1 \big)  \\ & \quad\times\sum_{j_1 \ne j} K_{j_1, j} \left| \frac{\lambda_j}{\lambda_{j_1, j}} \right|^{N_j^{(0)}} \big(N_j^{(0)}\big)^{\Re(\mu_{j_1}) - 1} \mathcal{O}(1)  \\ & \quad+ \lambda_j^{-z} \left( \frac{|\lambda_j|}{|\lambda_j|+\alpha_j^{(0)} } \right)^{\Re(z)}\sum_{j_1 \ne j} K_{j_1, j} \left( \frac{\alpha_j^{(0)}}{|\lambda_{j_1, j}|} \right)^{N_j^{(0)}} \big(N_j^{(0)}\big)^{\Re(\mu_{j_1}) - 1}  \mathcal{O}(1),
\end{split}
\end{gather}
which is valid for large $\Re(z)$ and $N_j^{(0)}$, provided that both gamma function arguments are positive. We now assume
\[
N_j^{(0)}  = \beta _j^{(0)} \Re(z) + \gamma _j^{(0)} ,
\]
where $\beta_j^{(0)} \in (0,1)$ is a parameter at our disposal, and $\gamma_j^{(0)}$ is bounded. Substituting this into \eqref{R0estimate} and applying Stirling's formula \cite[Eq.~\href{http://dlmf.nist.gov/5.11.E7}{(5.11.7)}]{DLMF} yields
\begin{align*}
R_j^{(0)}\big(z,\eta ;N_j^{(0)} \big) &=\lambda_j^{-z} \sum_{\substack{j_1 \ne j}} 
K_{j_1,j}  \Gamma\left( \Re\left(z + \mu_{j_1}\right) + \tfrac{1}{2} \right)
\big( 1 - \beta_j^{(0)} \big)^{(1 - \beta_j^{(0)}) \Re(z)}\\&\quad \times
\left( \frac{|\lambda_j|}{|\lambda_{j_1,j}|} \beta_j^{(0)} \right)^{\beta_j^{(0)}\Re(z)} \mathcal{O}(1) \\
&\quad + \lambda_j^{-z} 
\left( \frac{|\lambda_j|}{|\lambda_j| + \alpha_j^{(0)}} \right)^{\Re(z)} 
\sum_{\substack{j_1 \ne j}} K_{j_1,j}
\left( \frac{\alpha_j^{(0)}}{|\lambda_{j_1,j}|} \right)^{\beta_j^{(0)} \Re(z)} 
\left( \Re(z) \right)^{\Re(\mu_{j_1}) - 1} \mathcal{O}(1) \\
&= \lambda_j^{-z} \sum_{\substack{j_1 \ne j}} 
K_{j_1,j}  \Gamma\left( \Re\left(z + \mu_{j_1}\right) + \tfrac{1}{2} \right)
\big( 1 - \beta_j^{(0)} \big)^{(1 - \beta_j^{(0)}) \Re(z)}\\&\quad \times
\left( \frac{|\lambda_j|}{|\lambda_{j_1,j}|} \beta_j^{(0)} \right)^{\beta_j^{(0)} \Re(z)} \mathcal{O}(1),
\end{align*}
as $\Re(z)\to+\infty$. Each term in the final sum is minimised when
\[
1 - \beta_j^{(0)} = \frac{|\lambda_j|}{|\lambda_{j_1,j}|} \beta_j^{(0)} 
\; \Longleftrightarrow \; 
\beta_j^{(0)} = \frac{|\lambda_{j_1,j}|}{|\lambda_j| + |\lambda_{j_1,j}|}
\]
(see, for example, \cite[Theorem 2.7.1]{Cover1991}). This yields the optimal choice
\[
\beta _j^{(0)}  = \frac{\alpha _j^{(0)} }{\left| \lambda _j  \right| + \alpha _j^{(0)} }.
\]
With this choice, we obtain the main result of this subsection: if $\eta$ is such that $\eta^- \le \arg \lambda_j \le \eta^+$ and $N_j^{(0)}=\frac{\alpha _j^{(0)} }{\left| \lambda _j  \right| + \alpha _j^{(0)} } \Re(z)+\mathcal{O}(1)$, then the remainder term in \eqref{wremainder} satisfies
\[
R_j^{(0)}\big(z,\eta ;N_j^{(0)} \big)  = \lambda _j^{ - z} \Gamma \big( \Re (z) + \widetilde{\mu}  + \tfrac{1}{2} \big)\left( \frac{\left| \lambda _j \right|}{\left| \lambda _j  \right| + \alpha _j^{(0)} } \right)^{\Re (z)} \mathcal{O}(1),
\]
as $\Re(z)\to+\infty$.

\subsection{Level-one} In this section, we present a re-expansion of the remainder term $R_j^{(0)}$ in \eqref{wremainder} in terms of the hyperterminant $H^{(2)}$. This yields the level-one hyperasymptotic expansion, also called an exponentially improved asymptotic expansion.

First, we state a re-expansion of the remainder $\mathsf{R}_j^{(0)}$ in \eqref{Yremainder} in terms of the hyperterminant $F^{(1)}$, which follows from the procedure described in \cite[\S3]{OldeDaalhuis1998b}. Due to the reasons discussed in the previous subsection, the resulting expression is simpler than the corresponding one in \cite{OldeDaalhuis1998b}. We obtain
\begin{gather}\label{Ylevel1}
\begin{split}
\mathsf{R}_j^{(0)}\big(u, \widetilde\eta ; N_j^{(0)}\big) 
= \e^{\lambda_j u} u^{\mu_j - N_j^{(0)}} 
\sum_{j_1 \ne j} K_{j_1, j} &\sum_{s = 0}^{N_{j_1}^{(1)} - 1} 
a_{s, j_1} \left( \lambda_{j_1} \e^{- \pi \im} \right)^{\mu_{j_1} - s}\\ & \times
F^{(1)}\!\left(u; 
\mytop{N_j^{(0)} - s + \mu_{j_1, j}}{\lambda_{j_1, j}}
\right) 
+ \mathsf{R}_j^{(1)}(u, \widetilde\eta),
\end{split}
\end{gather}
with
\begin{align*}
\mathsf{R}_j^{(1)}(u, \widetilde{\eta}) & = 
\sum_{\substack{j_1 \ne j}} 
\sum_{\substack{j_2 \ne j_1}} 
\frac{K_{j_1,j} K_{j_2,j_1}}{2\pi \im} 
\int_{\lambda_j}^{[\widetilde{\eta}]} 
\int_{\lambda_{j_1}}^{[\theta_{j,j_1}]} 
\int_{\gamma_{j_2}(\theta_{j_1,j_2})} \frac{\e^{u t_0}}{\left(t_1 - t_0\right)\left(t_2 - t_1\right)}\\  & \quad \times
\left( \frac{t_0 - \lambda_j}{t_1 - \lambda_j} \right)^{N_j^{(0)} - \mu_j} 
\left( \frac{t_1 - \lambda_{j_1}}{t_2 - \lambda_{j_1}} \right)^{N_{j_1}^{(1)} - \mu_{j_1}} 
y_{j_2}(t_2, \widetilde\eta) 
\id t_2 \id t_1 \id t_0,
\end{align*}
valid under the conditions that $\arg u = \pi - \arg \lambda_j$, $N_j^{(0)}  > \max (0,\Re(\mu _j ) - 1,N_{j_1 }^{(1)}  - \Re(\mu _{j_1 ,j} ) )$, and $N_{j_1 }^{(1)}  > \max (0,\Re (\mu _{j_1 } ) - 1,\Re(\mu _{j_1 } ) + a)$, for all indices $j_1 \ne j$ such that $K_{j_1,j} \ne 0$. The contours $\gamma_{j_2}(\theta_{j_1,j_2})$ are defined analogously to those in \eqref{R0integral}. To keep the notation concise, the dependence of $\mathsf{R}_j^{(1)}$ on $N_j^{(0)}$ and the values of $N_{j_1}^{(1)}$ is suppressed.

Substituting this result into \eqref{remainderformula} yields the re-expansion
\begin{gather}\label{level1}
\begin{split}
R_j^{(0)} \big( z,\eta ;N_j^{(0)}  \big) = \lambda _j^{ - z} &\Gamma\big(z+\mu_j-N_j^{(0)}+1\big)\sum_{j_1 \ne j} K_{j_1 ,j} \sum_{s = 0}^{N_{j_1}^{(1)}  - 1} a_{s,j_1 } \\ & \times H^{(2)} \!\left( z;\mytop{\mu _j  - N_j^{(0)} ,}{\lambda _j,}
\mytop{N_j^{(0)}  - s + \mu _{j_1 ,j} }{\lambda _{j_1 ,j}} \right)  + R_j^{(1)} (z,\eta ),
\end{split}
\end{gather}
where the remainder term $R_j^{(1)}$ is given by
\begin{equation}\label{remainder1formula}
R_j^{(1)} (z,\eta ) = \e^{ - \pi \im z} \int_0^{\left[ \pi  - \arg \lambda _j  \right]} u^z \mathsf{R}_j^{(1)} (u,\widetilde \eta )\id u ,
\end{equation}
valid under the condition $\Re (z) > \max (N_j^{(0)}  - \Re(\mu _j)  - 1,a)$. As before, to keep the notation concise, the dependence of $R_j^{(1)}$ on $N_j^{(0)}$ and the values of $N_{j_1}^{(1)}$ is omitted.

By applying the method leading to the penultimate expression in \cite[Eq.~(4.13)]{OldeDaalhuis1998b}, we obtain the estimate
\begin{align*}
\mathsf{R}_j^{(1)}(u, \widetilde{\eta}) &= 
\e^{\lambda_j u} 
\left| u \right|^{\Re(\mu_j) - N_j^{(0)} + 1} \\ &\quad \times
\sum_{\substack{j_1 \ne j}} \sum_{\substack{j_2 \ne j_1}} K_{j_1,j}K_{j_2,j_1}
\Gamma\big(N_{j_1}^{(1)} - \Re\left(\mu_{j_1}\right) + 1\big)
\Gamma\big(N_j^{(0)} - N_{j_1}^{(1)} + \Re\left(\mu_{j_1,j}\right) - 1\big)
 \\ &\quad \times 
\left| \lambda_{j_1,j} \right|^{N_{j_1}^{(1)} - N_j^{(0)}}
\left| \lambda_{j_2,j_1} \right|^{-N_{j_1}^{(1)}}\left| u \right|^{\Re(\mu_{j_2})}
\mathcal{O}(1),
\end{align*}
which holds for large values of $u$, $N_j^{(0)}$, and the $N_{j_1}^{(1)}$, with $\arg u = \pi - \arg \lambda_j$, and assuming that the arguments of the gamma functions are positive. Moreover, using \eqref{Ysmall} and the boundedness of the hyperterminants $F^{(1)}$ near the origin, we conclude that this estimate remains valid for all $u$ satisfying $\arg u = \pi - \arg \lambda_j$, possibly after adjusting the implied constants. Substituting this result into \eqref{remainder1formula} yields
\begin{gather}\label{R1estimate}
\begin{split}
R_j^{(1)}(z, \eta) & = 
\lambda_j^{-z} 
\sum_{\substack{j_1 \ne j}}\sum_{\substack{j_2 \ne j_1}} K_{j_1,j} K_{j_2,j_1}  \Gamma\big( \Re\left( z + \mu_j + \mu_{j_2} \right) - N_j^{(0)} + 2 \big)
\Gamma\big( N_{j_1}^{(1)} - \Re\left( \mu_{j_1} \right) + 1 \big)
 \\ &\quad \times \Gamma\big( N_j^{(0)} - N_{j_1}^{(1)} + \Re\left( \mu_{j_1,j} \right) - 1 \big)\left| \lambda_j \right|^{N_j^{(0)}}
\left| \lambda_{j_1,j} \right|^{N_{j_1}^{(1)} - N_j^{(0)}}
\left| \lambda_{j_2,j_1} \right|^{-N_{j_1}^{(1)}} 
\mathcal{O}(1),
\end{split}
\end{gather}
which is valid for large $\Re(z)$, $N_j^{(0)}$, and all $N_{j_1}^{(1)}$, provided that all gamma function arguments are positive. We now aim to optimise the remainder $R_j^{(1)}$ by setting
\[
N_j^{(0)}  = \beta _j^{(0)} \Re(z) + \gamma _j^{(0)} ,\quad N_{j_1 }^{(1)}  = \beta _{j_1 }^{(1)} \Re(z) + \gamma _{j_1 }^{(1)} ,
\]
where the constants $\beta _j^{(0)}$ and $\beta _{j_1 }^{(1)}$ satisfy
$0<\beta _{j_1 }^{(1)}<\beta _j^{(0)}<1$, and $\gamma _j^{(0)}$ and $\gamma _{j_1 }^{(1)}$ remain bounded as $\Re(z)\to+\infty$. Substituting these expressions into \eqref{R1estimate} and applying Stirling's formula gives
\begin{align*}
R_j^{(1)}(z, \eta) & = 
\lambda_j^{-z} 
\sum_{\substack{j_1 \ne j}} 
\sum_{\substack{j_2 \ne j_1}} 
K_{j_1,j} K_{j_2,j_1} \Gamma\left( \Re\left( z + \mu_{j_2} \right) + 1 \right) \big(1 - \beta_j^{(0)}\big)^{(1 - \beta_j^{(0)}) \Re(z)} \\ & \quad \times 
\left( 
  \frac{|\lambda_j|}{|\lambda_{j_1,j}|} \big( \beta_j^{(0)} - \beta_{j_1}^{(1)} \big) 
\right)^{(\beta_j^{(0)} - \beta_{j_1}^{(1)}) \Re(z)} 
\left( 
  \frac{|\lambda_j|}{|\lambda_{j_2,j_1}|} \beta_{j_1}^{(1)} 
\right)^{\beta_{j_1}^{(1)} \Re(z)}  
\mathcal{O}(1),
\end{align*}
as $\Re(z)\to+\infty$. Each term on the right-hand side is minimised when
\begin{align*}
1 - \beta_j^{(0)} &= \frac{|\lambda_j|}{|\lambda_j| + |\lambda_{j_1,j}| + |\lambda_{j_2,j_1}|}, \quad
\beta_j^{(0)} - \beta_{j_1}^{(1)} = \frac{|\lambda_{j_1,j}|}{|\lambda_j| + |\lambda_{j_1,j}| + |\lambda_{j_2,j_1}|}, \\
\beta_{j_1}^{(1)} &= \frac{|\lambda_{j_2,j_1}|}{|\lambda_j| + |\lambda_{j_1,j}| + |\lambda_{j_2,j_1}|}
\end{align*}
(see again, for instance, \cite[Theorem 2.7.1]{Cover1991}). This leads to the optimal choice
\[
1-\beta _j^{(0)}  = \frac{\left| \lambda _j  \right|}{\left| \lambda _j  \right| + \alpha _j^{(1)} } \;\Longleftrightarrow\; \beta _j^{(0)}  = \frac{\alpha _j^{(1)} }{\left| \lambda _j  \right| + \alpha _j^{(1)} },
\]
where $\alpha_j^{(1)}$ is defined in \eqref{alpha}. For each $j_1$ such that $K_{j_1,j} \ne 0$, we may optimise the value of $\beta_{j_1}^{(1)}$ by selecting $j_2$ with $K_{j_2,j_1} \ne 0$ appropriately. Nevertheless, a reasonable and effective choice is
\[
\beta _{j_1 }^{(1)}  = \beta _j^{(0)}  - \frac{\alpha _j^{(0)} }{\left| \lambda _j  \right| + \alpha _j^{(1)} } =\frac{\alpha _j^{(1)}-\alpha _j^{(0)} }{\left| \lambda _j  \right| + \alpha _j^{(1)} }.
\]
With these choices, we arrive at the final estimate for the level-one hyperasymptotic remainder in \eqref{level1}:
\[
R_j^{(1)}(z, \eta) = 
\lambda_j^{-z} \Gamma\left( \Re(z) + \widetilde{\mu} + 1 \right) 
\left( \frac{ \left| \lambda_j \right| }{ \left| \lambda_j \right| + \alpha_j^{(1)} } \right)^{\Re(z)} 
\mathcal{O}(1),
\]
as $\Re(z)\to+\infty$, provided that $\eta$ satisfies $\eta^- \le \arg \lambda_j \le \eta^+$.

\subsection{Level-two} In this section, we re-expand the remainder term $R_j^{(1)}$ from \eqref{level1} using the hyperterminant $H^{(3)}$, thereby obtaining the level-two hyperasymptotic expansion.

Analogously to the level-one case, where we begin by re-expanding the remainder $\mathsf{R}_j^{(0)}$ in terms of the hyperterminant $F^{(1)}$, we now obtain a re-expansion of the remainder $\mathsf{R}_j^{(1)}$ in \eqref{Ylevel1} using the hyperterminant $F^{(2)}$, following the procedure described in \cite[\S3]{OldeDaalhuis1998b}. For the reasons discussed in the previous subsections, the resulting expression is simpler than the corresponding one in \cite{OldeDaalhuis1998b}. We obtain
\begin{align*}
\mathsf{R}_j^{(1)}(u, \widetilde\eta) &= \e^{\lambda_j u} u^{\mu_j - N_j^{(0)}} 
\sum_{j_1 \ne j} \sum_{j_2 \ne j_1} 
K_{j_1, j} K_{j_2, j_1} 
\sum_{s = 0}^{N_{j_2}^{(2)} - 1} 
a_{s, j_2} \left( \lambda_{j_2} \e^{-\pi \im} \right)^{\mu_{j_2} - s} \\
&\quad \times F^{(2)}\!\left(u; 
\mytop{N_j^{(0)} - N_{j_1}^{(1)} + \mu_{j_1, j} + 1,}{\lambda_{j_1, j},}
\mytop{N_{j_1}^{(1)} - s + \mu_{j_2, j_1}}{\lambda_{j_2, j_1}}
\right)
+ \mathsf{R}_j^{(2)}(u, \widetilde\eta),
\end{align*}
with
\begin{align*}
\mathsf{R}_j^{(2)}(u, \widetilde{\eta}) &= 
\sum_{j_1 \ne j} \sum_{j_2 \ne j_1} \sum_{j_3 \ne j_2}
\frac{K_{j_1, j} K_{j_2, j_1} K_{j_3, j_2}}{2\pi \im}
\int_{\lambda_j}^{[\widetilde{\eta}]}
\int_{\lambda_{j_1}}^{[\theta_{j, j_1}]}
\int_{\lambda_{j_2}}^{[\theta_{j_1, j_2}]}
\int_{\gamma_{j_2}(\theta_{j_2,j_3})} \\&\quad \times
\frac{\e^{u t_0}}{\left(t_1 - t_0\right)\left(t_2 - t_1\right)\left(t_3 - t_2\right)} 
\left( \frac{t_0 - \lambda_j}{t_1 - \lambda_j} \right)^{N_j^{(0)} - \mu_j}
\left( \frac{t_1 - \lambda_{j_1}}{t_2 - \lambda_{j_1}} \right)^{N_{j_1}^{(1)} - \mu_{j_1}} \\&\quad \times\left( \frac{t_2 - \lambda_{j_2}}{t_3 - \lambda_{j_2}} \right)^{N_{j_2}^{(2)} - \mu_{j_2}}
y_{j_3}(t_3, \widetilde{\eta}) \id t_3 \id t_2 \id t_1 \id t_0,
\end{align*}
valid under the conditions that $\arg u=\pi-\arg\lambda_j$, $N_j^{(0)}  > \max (0,\Re(\mu _j ) - 1,N_{j_1 }^{(1)}  - \Re (\mu _{j_1 ,j} ))$, $N_{j_1 }^{(1)}  > \max (0,\Re(\mu _{j_1 } ) - 1,N_{j_2 }^{(2)}  - \Re (\mu _{j_2 ,j_1 } ))$, $N_{j_2 }^{(2)}  > \max (0,\Re(\mu _{j_2 } ) - 1,\Re(\mu _{j_2 } ) + a)$, for all indices $j_1 \ne j$ and $j_2 \ne j_1$ such that $K_{j_1,j} \ne 0$ and $K_{j_2,j_1} \ne 0$. The contours $\gamma_{j_2}(\theta_{j_2,j_3})$ are defined analogously to those in \eqref{R0integral}. For brevity, the dependence of $\mathsf{R}_j^{(2)}$ on $N_j^{(0)}$ as well as on the values of $N_{j_1}^{(1)}$ and $N_{j_2}^{(2)}$ is suppressed.

Substituting this result into \eqref{remainder1formula}, we obtain the re-expansion
\begin{gather}\label{level2}
\begin{split}
R_j^{(1)}(z, \eta) 
&= \lambda_j^{-z} \Gamma\big(z + \mu_j - N_j^{(0)} + 1\big) 
\sum_{j_1 \ne j} \sum_{j_2 \ne j_1} 
K_{j_1, j}  K_{j_2, j_1} 
\sum_{s = 0}^{N_{j_2}^{(2)} - 1} 
a_{s, j_2} \\
&\quad \times H^{(3)}\!\left(z;
\mytop{\mu_j - N_j^{(0)},}{\lambda_j}
\mytop{N_j^{(0)} - N_{j_1}^{(1)} + \mu_{j_1, j} + 1,}{\lambda_{j_1, j}}
\mytop{N_{j_1}^{(1)} - s + \mu_{j_2, j_1}}{\lambda_{j_2, j_1}}
\right) 
+ R_j^{(2)}(z, \eta),
\end{split}
\end{gather}
where the remainder term $R_j^{(2)}(z, \eta)$ is given by
\begin{equation}\label{remainder2formula}
R_j^{(2)} (z,\eta ) = \e^{ - \pi \im z} \int_0^{\left[ \pi  - \arg \lambda _j \right]} u^z \mathsf{R}_j^{(2)} (u,\widetilde{\eta})\id u,
\end{equation}
and is valid under the condition $\Re (z) > \max (N_j^{(0)}  - \Re(\mu _j)  - 1,a)$. As before, to keep the notation concise, we omit the explicit dependence of $R_j^{(2)}$ on $N_j^{(0)}$ as well as on the values of $N_{j_1}^{(1)}$ and $N_{j_2}^{(2)}$.

By applying the argument that leads to the second-to-last expression in \cite[Eq.~(5.11)]{OldeDaalhuis1998b}, we obtain the following estimate:
\begin{align*}
\mathsf{R}_j^{(2)}(u, \widetilde{\eta}) & =\e^{\lambda_j u} \left| u \right|^{\Re(\mu_j) - N_j^{(0)} + 1} \\ &\quad\times \sum_{j_1 \ne j} \sum_{j_2 \ne j_1} \sum_{j_3 \ne j_2} 
K_{j_1, j} K_{j_2, j_1} K_{j_3, j_2} \big( N_j^{(0)} - \Re\left(\mu_j\right) \big)
\big( N_{j_1}^{(1)} - \Re\left(\mu_{j_1}\right) \big)  \\ &\quad \times \Gamma\big( N_j^{(0)} - N_{j_1}^{(1)} + \Re\left(\mu_{j_1, j}\right) - 1 \big)   \Gamma\big( N_{j_2}^{(2)} - \Re\left(\mu_{j_2}\right) + 1 \big)  \\ &\quad\times \Gamma\big( N_{j_1}^{(1)} - N_{j_2}^{(2)} + \Re\left(\mu_{j_2, j_1}\right) - 1 \big)   \left| \lambda_{j_1, j} \right|^{N_{j_1}^{(1)} - N_j^{(0)}} 
\left| \lambda_{j_2, j_1} \right|^{N_{j_2}^{(2)} - N_{j_1}^{(1)}} \\ &\quad\times
\left| \lambda_{j_3, j_2} \right|^{-N_{j_2}^{(2)}} 
\left| u \right|^{\Re(\mu_{j_3})} 
\mathcal{O}(1),
\end{align*}
which holds for large values of $u$, $N_j^{(0)}$, and all $N_{j_1}^{(1)}$ and $N_{j_2}^{(2)}$, under the condition $\arg u = \pi - \arg \lambda_j$, and provided that the arguments of the gamma functions are positive. Furthermore, by invoking \eqref{Ysmall} and the boundedness of the hyperterminants $F^{(2)}$ near the origin, we conclude that the estimate remains valid for all values of $u$ with $\arg u = \pi - \arg \lambda_j$, possibly after adjusting the implied constants. Substituting this result into \eqref{remainder2formula}, we obtain
\begin{gather}\label{R2estimate}
\begin{split}
R_j^{(2)}(z, \eta) 
&= \lambda_j^{-z} 
\sum_{j_1 \ne j} \sum_{j_2 \ne j_1} \sum_{j_3 \ne j_2}
K_{j_1, j} K_{j_2, j_1} K_{j_3, j_2} 
\big( N_j^{(0)} - \Re\left(\mu_j\right) \big)
\big( N_{j_1}^{(1)} - \Re\left(\mu_{j_1}\right) \big) \\
&\quad \times \Gamma\big(  \Re\left(z + \mu_j + \mu_{j_3}\right) - N_j^{(0)} + 2 \big)
\Gamma\big( N_j^{(0)} - N_{j_1}^{(1)} + \Re\left(\mu_{j_1, j}\right) - 1 \big) \\
&\quad \times \Gamma\big( N_{j_2}^{(2)} - \Re\left(\mu_{j_2}\right) + 1 \big)
\Gamma\big( N_{j_1}^{(1)} - N_{j_2}^{(2)} + \Re\left(\mu_{j_2, j_1}\right) - 1 \big) \\
&\quad \times \left| \lambda_j \right|^{N_j^{(0)}}
\left| \lambda_{j_1, j} \right|^{N_{j_1}^{(1)} - N_j^{(0)}}
\left| \lambda_{j_2, j_1} \right|^{N_{j_2}^{(2)} - N_{j_1}^{(1)}}
\left| \lambda_{j_3, j_2} \right|^{-N_{j_2}^{(2)}}
\mathcal{O}(1),
\end{split}
\end{gather}
which is valid for large $\Re(z)$, $N_j^{(0)}$, and all $N_{j_1}^{(1)}$ and $N_{j_2}^{(2)}$, provided that all gamma function arguments are positive. To optimise the remainder $R_j^{(2)}$, we now set
\[
N_j^{(0)} = \beta_j^{(0)} \Re(z) + \gamma_j^{(0)}, \quad
N_{j_1}^{(1)} = \beta_{j_1}^{(1)} \Re(z) + \gamma_{j_1}^{(1)}, \quad
N_{j_2}^{(2)} = \beta_{j_2}^{(2)} \Re(z) + \gamma_{j_2}^{(2)},
\]
where the constants $\beta _j^{(0)}$, $\beta _{j_1 }^{(1)}$, and $\beta _{j_2 }^{(2)}$ satisfy
$0<\beta _{j_2 }^{(2)}<\beta _{j_1 }^{(1)}<\beta _j^{(0)}<1$, and the constants $\gamma _j^{(0)}$, $\gamma _{j_1 }^{(1)}$, and $\gamma _{j_2 }^{(2)}$ remain bounded as $\Re(z)\to+\infty$. Substituting these expressions into \eqref{R2estimate} and using Stirling’s formula yields
\begin{align*}
R_j^{(2)}(z, \eta)
&= \lambda_j^{-z} 
\sum_{j_1 \ne j} \sum_{j_2 \ne j_1} \sum_{j_3 \ne j_2}
K_{j_1, j} K_{j_2, j_1} K_{j_3, j_2} 
\Gamma\left( \Re\left(z + \mu_{j_3}\right) + \tfrac{3}{2} \right) \\
&\quad \times \big( 1 - \beta_j^{(0)} \big)^{(1 - \beta_j^{(0)}) \Re(z)} \left( 
\frac{|\lambda_j|}{|\lambda_{j_1, j}|} 
\big( \beta_j^{(0)} - \beta_{j_1}^{(1)} \big)
\right)^{( \beta_j^{(0)} - \beta_{j_1}^{(1)} ) \Re(z)} \\
&\quad \times \left( 
\frac{|\lambda_j|}{|\lambda_{j_2, j_1}|} 
\big( \beta_{j_1}^{(1)} - \beta_{j_2}^{(2)} \big)
\right)^{( \beta_{j_1}^{(1)} - \beta_{j_2}^{(2)} ) \Re(z)} \left( 
\frac{|\lambda_j|}{|\lambda_{j_3, j_2}|} 
\beta_{j_2}^{(2)}
\right)^{\beta_{j_2}^{(2)} \Re(z)}
\mathcal{O}(1),
\end{align*}
as $\Re(z)\to+\infty$. Each term on the right-hand side is minimised when
\begin{align*}
1 - \beta_j^{(0)} &= \frac{|\lambda_j|}{|\lambda_j| + |\lambda_{j_1, j}| + |\lambda_{j_2, j_1}| + |\lambda_{j_3, j_2}|} ,\quad
\beta_j^{(0)} - \beta_{j_1}^{(1)} = \frac{|\lambda_{j_1, j}|}{|\lambda_j| + |\lambda_{j_1, j}| + |\lambda_{j_2, j_1}| + |\lambda_{j_3, j_2}|}, \\
\beta_{j_1}^{(1)} - \beta_{j_2}^{(2)} &= \frac{|\lambda_{j_2, j_1}|}{|\lambda_j| + |\lambda_{j_1, j}| + |\lambda_{j_2, j_1}| + |\lambda_{j_3, j_2}|}, \quad
\beta_{j_2}^{(2)} = \frac{|\lambda_{j_3, j_2}|}{|\lambda_j| + |\lambda_{j_1, j}| + |\lambda_{j_2, j_1}| + |\lambda_{j_3, j_2}|}.
\end{align*}
Once again, we find that the optimal choice for $\beta_j^{(0)}$ is given by
\[
1 - \beta_j^{(0)} = \frac{|\lambda_j|}{|\lambda_j| + \alpha_j^{(2)}}
 \;\Longleftrightarrow \;
\beta_j^{(0)} = \frac{\alpha_j^{(2)}}{|\lambda_j| + \alpha_j^{(2)}},
\]
where $\alpha_j^{(2)}$ is defined in \eqref{alpha}. A reasonable and effective choice for the constants $\beta_{j_1}^{(1)}$ and $\beta_{j_2}^{(2)}$ is
\[
\beta_{j_1}^{(1)} = \beta_j^{(0)} - \frac{\alpha_j^{(0)}}{|\lambda_j| + \alpha_j^{(2)}}
= \frac{\alpha_j^{(2)} - \alpha_j^{(0)}}{|\lambda_j| + \alpha_j^{(2)}},
\quad
\beta_{j_2}^{(2)} = \beta_{j_1}^{(1)} - \frac{\alpha_j^{(1)} - \alpha_j^{(0)}}{|\lambda_j| + \alpha_j^{(2)}}
= \frac{\alpha_j^{(2)} - \alpha_j^{(1)}}{|\lambda_j| + \alpha_j^{(2)}}.
\]
With these choices, the final estimate for the level-two hyperasymptotic remainder in \eqref{level2} becomes
\[
R_j^{(2)} (z,\eta ) = \lambda _j^{ - z} \Gamma \left( \Re(z) + \widetilde{\mu}  + \tfrac{3}{2} \right)\left( \frac{\left| \lambda _j \right|}{\left| \lambda _j \right| + \alpha _j^{(2)} } \right)^{\Re(z)} \mathcal{O}(1),
\]
as $\Re(z) \to +\infty$, provided that $\eta$ satisfies $\eta^- \leq \arg \lambda_j \leq \eta^+$.

\subsection{General level} The method for deriving the level-$\ell$ hyperasymptotic expansion is now clear from the cases $\ell = 0, 1, 2$. The general formula can be directly stated by analogy and rigorously confirmed through induction.

\begin{theorem}\label{thm1} Let $\ell$ be an arbitrary non-negative integer, and let the integers $N_j^{(0)}, N_{j_1}^{(1)}, \ldots, N_{j_\ell}^{(\ell)}$ satisfy
\[
N_j^{(0)}=\beta_j^{(0)} \Re(z)+\gamma_j^{(0)}, \quad N_{j_r}^{(r)}=\beta_{j_r}^{(r)} \Re(z)+\gamma_{j_r}^{(r)}, \quad r=1,2,\ldots,\ell,
\]
where the constants $\beta$ satisfy
\[
0<\beta_{j_\ell}^{(\ell)}<\beta_{j_{\ell-1}}^{(\ell-1)}<\cdots<\beta_{j_1}^{(1)}<\beta_j^{(0)}<1,
\]
and each $\gamma$ remains bounded as $\Re(z)\to+\infty$. Let $\eta \in \mathbb{R}$ be an admissible direction such that $|\eta - \arg \lambda_j| < \pi$ and $\eta^- \le \arg \lambda_j \le \eta^+$. Then, for each $1 \le j \le n$, the solution of the difference equation \eqref{diffeq} defined in Proposition \ref{prop3} admits a level-$\ell$ hyperasymptotic expansion of the form
\begin{align*}
w_j & (z,\eta)=\lambda_j^{ - z} \Bigg( \sum_{s = 0}^{N_j^{(0)} - 1} a_{s,j} \Gamma \left( z + \mu _j - s \right) \Bigg.\\& +\Gamma\big(z+\mu_j-N_j^{(0)}+1\big)\Bigg[
\sum_{j_1 \ne j} K_{j_1,j} \Bigg\{ \sum_{s = 0}^{N_{j_1}^{(1)} - 1} a_{s,j_1} H^{(2)}\!\left(z;\mytop{\mu_j - N_j^{(0)},}{\lambda_j,} \mytop{N_j^{(0)} - s + \mu_{j_1,j}}{\lambda_{j_1,j}} \right) \Bigg.\Bigg. \\& + \sum_{j_2 \ne j_1} K_{j_2,j_1} \Bigg\{ \sum_{s = 0}^{N_{j_2}^{(2)} - 1} a_{s,j_2} H^{(3)}\!\left(z;\mytop{\mu_j - N_j^{(0)},}{\lambda_j,} \mytop{N_j^{(0)} - N_{j_1}^{(0)} + \mu_{j_1,j} + 1,}{\lambda_{j_1,j},} \mytop{N_{j_1}^{(0)} - s + \mu_{j_2,j_1}}{\lambda_{j_2,j_1}} \right) \Bigg. \\& \ddots \\& + \sum_{j_\ell \ne j_{\ell - 1}} K_{j_\ell,j_{\ell - 1}} \Bigg\{ \sum_{s = 0}^{N_{j_\ell}^{(\ell)} - 1} a_{s,j_\ell} H^{(\ell + 1)}\!\left(z;\mytop{\mu_j - N_j^{(0)},}{\lambda_j,} \mytop{N_j^{(0)} - N_{j_1}^{(0)} + \mu_{j_1,j} + 1,}{\lambda_{j_1,j},} \mytop{\ldots,}{\ldots,}\right.\Bigg. \\& \phantom{+ \sum_{j_\ell \ne j_{\ell - 1}} \frac{K_{j_\ell,j_{\ell - 1}}}{2\pi \im}\left\{\right.}\left.\left. \mytop{N_{j_{\ell - 2}}^{(\ell - 2)} - N_{j_{\ell - 1}}^{(\ell - 1)} + \mu_{j_{\ell - 1},j_{\ell - 2}} + 1,}{\lambda_{j_{\ell - 1},j_{\ell - 2}},} \mytop{N_{j_{\ell - 1}}^{(\ell - 1)} - s + \mu_{j_\ell,j_{\ell - 1}}}{\lambda_{j_\ell,j_{\ell - 1}}}\right) \right. \\& \Bigg. \Bigg. \Bigg. \Bigg. {} \Bigg\} \cdots \Bigg\} \Bigg\} \Bigg] \Bigg) + R_j^{(\ell)}(z, \eta),
\end{align*}
as $\Re(z) \to +\infty$, where the remainder term $R_j^{(\ell)}$ satisfies the bound
\begin{align*}
R_j^{(\ell)}(z, \eta) & =\lambda_j^{-z} 
\sum_{j_1 \neq j } \cdots
\sum_{j_{\ell+1} \neq j_\ell} 
K_{j_1,j} \cdots K_{j_{\ell+1}, j_\ell}  
\Gamma\left(\Re\left(z + \mu_{j_{\ell+1}}\right) + \tfrac{\ell+1}{2}\right) \\ & \quad \times
\big(1 - \beta_j^{(0)}\big)^{(1 - \beta_j^{(0)}) \Re(z)}
\left(
\frac{|\lambda_j|}{|\lambda_{j_1,j}|} \big(\beta_j^{(0)} - \beta_{j_1}^{(1)}\big)
\right)^{(\beta_j^{(0)} - \beta_{j_1}^{(1)}) \Re(z)}\\ & \quad 
\cdots
\left(
\frac{|\lambda_j|}{|\lambda_{j_\ell, j_{\ell-1}}|} \big(\beta_{j_{\ell-1}}^{(\ell-1)} - \beta_{j_\ell}^{(\ell)}\big)
\right)^{(\beta_{j_{\ell-1}}^{(\ell-1)} - \beta_{j_\ell}^{(\ell)}) \Re(z)}
\left(
\frac{|\lambda_j|}{|\lambda_{j_{\ell+1}, j_\ell}|} \beta_{j_\ell}^{(\ell)}
\right)^{\beta_{j_\ell}^{(\ell)} \Re(z)} 
\mathcal{O}(1),
\end{align*}
as $\Re(z)\to +\infty$. In the case $\ell = 0$, the notation is understood as $j_0 = j$.
\end{theorem}

\begin{remark}
 As explained in \cite{OldeDaalhuis1998b}, if none of the triples
\[
\left\{ \lambda _{r_1 } ,\lambda _{r_2 } ,\lambda _{r_3 } \mid \lambda _{r_1 }  \ne \lambda _{r_2 }  \ne \lambda _{r_3 }  \ne \lambda _{r_1 }  \right\}
\]
are collinear, then the estimate for the remainder $\mathsf{R}_j^{(\ell)}$ in the level-$\ell$ hyperasymptotic expansion of $Y_j(u,\widetilde{\eta})$ can be improved by a factor of $|u|^{-\ell - 1/2}$. Accordingly, the shift $\frac{\ell + 1}{2}$ in the gamma function appearing in the estimate for $R_j^{(\ell)}$ can be replaced by $-\frac{\ell}{2}$.
\end{remark}

The remainder terms $R_j^{(\ell)}$ in Theorem \ref{thm1} are not optimised. Optimisation is given by the following theorem:

\begin{theorem}\label{thm3} Let
\[
\beta_j^{(0)} = \frac{\alpha_j^{(\ell)}}{|\lambda_j| + \alpha_j^{(\ell)}},\quad
\beta_{j_1}^{(1)} = \frac{\alpha_j^{(\ell)} - \alpha_j^{(0)}}{|\lambda_j| + \alpha_j^{(\ell)}},\quad
\beta_{j_2}^{(2)} = \frac{\alpha_j^{(\ell)} - \alpha_j^{(1)}}{|\lambda_j| + \alpha_j^{(\ell)}}, \ldots,\quad
\beta_{j_\ell}^{(\ell)} = \frac{\alpha_j^{(\ell)} - \alpha_j^{(\ell-1)}}{|\lambda_j| + \alpha_j^{(\ell)}},
\]
where $\alpha_j^{(m)}$ is defined by \eqref{alpha}. Then, as $\Re(z)\to+\infty$,
\begin{equation}\label{genlevelremainder}
R_j^{(\ell )} (z,\eta ) = \lambda _j^{ - z} \Gamma \left( \Re(z) + \widetilde{\mu}  + \tfrac{\ell  + 1}{2} \right)\left( \frac{\left| \lambda _j \right|}{\left| \lambda _j \right| + \alpha _j^{(\ell )} } \right)^{\Re(z)} \mathcal{O}(1).
\end{equation}
If none of the triples
\[
\left\{ \lambda _{r_1 } ,\lambda _{r_2 } ,\lambda _{r_3 } \mid \lambda _{r_1 }  \ne \lambda _{r_2 }  \ne \lambda _{r_3 }  \ne \lambda _{r_1 }  \right\}
\]
are collinear, then the sharper estimate
\[
R_j^{(\ell )} (z,\eta ) = \lambda _j^{ - z} \Gamma \left( \Re(z) + \widetilde{\mu}  - \tfrac{\ell}{2} \right)\left( \frac{\left| \lambda _j \right|}{\left| \lambda _j \right| + \alpha _j^{(\ell )} } \right)^{\Re(z)} \mathcal{O}(1)
\]
also holds, as $\Re(z)\to+\infty$.
\end{theorem}

Since
\[
\alpha_j^{(\ell)} \ge (\ell + 1) \min_{p \ne q} |\lambda_p-\lambda_q| \to +\infty
\]
as $\ell \to +\infty$, it follows from \eqref{genlevelremainder} that for sufficiently large $\ell$,
\[
R_j^{(\ell)}(z, \eta) = o\big(\lambda_r^{-z}  \Gamma\left(z + \mu_r\right)\big)
\]
as $\Re(z) \to +\infty$ for every $1 \le r \le n$. In other words, for such $\ell$, the solution $w_j(z, \eta)$ is uniquely determined by its level-$\ell$ hyperasymptotic expansion.

\section{Computation of the connection coefficients}\label{sec4}

To make Theorem \ref{thm1} practically useful, it is necessary to compute the connection coefficients $K_{\ell,j}$ defined in \eqref{connection}. In \cite{OldeDaalhuis1998b}, it was shown that for differential equations, such coefficients can be computed numerically to arbitrary precision by applying a hyperasymptotic expansion to the late terms in the associated asymptotic power series. In our setting, we apply the same idea to the coefficients appearing in the inverse factorial expansions \eqref{winvfact}. Combining \eqref{wremainder} with Theorem \ref{thm1}, we obtain {\allowdisplaybreaks
\begin{align}
a_{N_j^{(0)}, j} & = \frac{R_j^{(0)}\big(z, \eta; N_j^{(0)}\big) - R_j^{(0)}\big(z, \eta; N_j^{(0)} + 1\big)}{\Gamma\big(z + \mu_j - N_j^{(0)}\big)} \nonumber \\
&= \sum_{j_1 \ne j} K_{j_1, j} \Bigg\{ 
\left( \frac{\lambda_j}{\lambda_{j_1,j}} \right)^{N_j^{(0)} - \mu_j} 
\sum_{s = 0}^{N_{j_1}^{(1)} - 1} 
a_{s, j_1} 
\left( \frac{\lambda_{j_1}}{\lambda_{j_1,j}} \right)^{\mu_{j_1} - s} 
\Gamma \big( N_j^{(0)} + \mu_{j_1,j} - s \big) \nonumber \\
&\quad + \Gamma \big( N_j^{(0)} - N_{j_1}^{(1)} + \mu_{j_1,j} + 1 \big)
\sum_{j_2 \ne j_1} K_{j_2, j_1} \Bigg\{
\left( \frac{\lambda_j}{\lambda_{j_2,j}} \right)^{N_j^{(0)} - \mu_j}
\sum_{s = 0}^{N_{j_2}^{(2)} - 1}
a_{s, j_2}
\left( \frac{\lambda_{j_2}}{\lambda_{j_2,j}} \right)^{\mu_{j_2} - s} \nonumber \\
&\hspace{14.5em} \times H^{(2)}\!\left( 0; 
\mytop{N_j^{(0)} - N_{j_1}^{(1)} + \mu_{j_1,j},}{\lambda_{j_1,j},}
\mytop{N_{j_1}^{(1)} - s + \mu_{j_2,j_1}}{\lambda_{j_2,j_1}}
\right) \Bigg\} \nonumber \\ 
&\ddots \label{ahyper} \\
&\quad + \sum_{j_\ell \ne j_{\ell - 1}} K_{j_\ell, j_{\ell - 1}} \Bigg\{
\left( \frac{\lambda_j}{\lambda_{j_\ell,j}} \right)^{N_j^{(0)} - \mu_j}
\sum_{s = 0}^{N_{j_\ell}^{(\ell)} - 1}
a_{s, j_\ell}
\left( \frac{\lambda_{j_\ell}}{\lambda_{j_\ell,j}} \right)^{\mu_{j_\ell} - s} \nonumber\\
&\hspace{14.5em} \times H^{(\ell)}\!\left( 0; 
\mytop{N_j^{(0)} - N_{j_1}^{(1)} + \mu_{j_1,j},}{\lambda_{j_1,j},}
\mytop{\ldots,}{\ldots,} \right. \nonumber\\ & \hspace{11.7em} \left.
\mytop{N_{j_{\ell - 2}}^{(\ell - 2)} - N_{j_{\ell - 1}}^{(\ell - 1)} + \mu_{j_{\ell - 1}, j_{\ell - 2}} + 1,}{\lambda_{j_{\ell - 1}, j_{\ell - 2}},}
\mytop{N_{j_{\ell - 1}}^{(\ell - 1)} - s + \mu_{j_\ell, j_{\ell - 1}}}{\lambda_{j_\ell, j_{\ell - 1}}}
\right) \nonumber\\ &\quad  \Bigg\} \cdots \Bigg\}
\Bigg\} + r_j^{(\ell)} \big( N_j^{(0)} \big), \nonumber
\end{align}}
where the remainder term $r_j^{(\ell)}$ satisfies
\begin{align*}
r_j^{(\ell)} \big( N_j^{(0)} \big) & =\e^{ - \beta _j^{(0)} z} \big( \big(1 - \beta _j^{(0)} \big)z \big)^{N_j^{(0)}  - \Re(\mu _j ) + \widetilde{\mu}  + \frac{\ell  + 1}{2}} \mathcal{O}(1) \\ & =\left( \frac{|\lambda_j|}{\alpha_j^{(\ell)}} \right)^{N_j^{(0)}}
\Gamma\big( N_j^{(0)} - \Re(\mu_j) + \widetilde{\mu} + \tfrac{\ell}{2} + 1 \big) \mathcal{O}(1)
\end{align*}
provided that $N_j^{(0)}=\beta_j^{(0)} z+\mathcal{O}(1)$ and $N_{j_r}^{(0)}=\beta_{j_r}^{(r)} z+\mathcal{O}(1)=(\beta_{j_r}^{(r)}/\beta_{j}^{(0)})N_j^{(0)}+\mathcal{O}(1)$ for $r=1,2,\ldots,\ell$, as $z\to+\infty$, where $\beta_j^{(0)}$ and $\beta_{j_r}^{(r)}$ are given in Theorem \ref{thm3}.

\begin{remark}
Although the expansion \eqref{ahyper} contains terms where division by $\lambda_{j_r,j}$ occurs for $r = 1,2, \ldots, \ell$, these divisions never actually lead to a singularity. Indeed, in the definition of the functions $H^{(r)}$, the same factors appear with opposite powers, which exactly cancel the potentially singular terms. Therefore, \eqref{ahyper} remains well-defined for all values of $\lambda_{j_r,j}$.
\end{remark}

The expression \eqref{ahyper} may be viewed as a hyperasymptotic expansion for the coefficients $a_{N_j^{(0)}, j}$ in the limit $N_j^{(0)} \to +\infty$. For second-order difference equations, the case $\ell = 1$ was also established in \cite{OldeDaalhuis2004,Olver2000}.

An alternative approach to obtaining this expansion is to reinterpret the coefficients $a_{s,j}$ as those appearing in the large-$u$ asymptotic expansion of the Laplace-type integral defining the function $Y_j(u,\eta)$. One then invokes the known hyperasymptotic formula for such integrals’ coefficients (cf. \cite[Eq.~(7.4)]{OldeDaalhuis1998b}), reproducing exactly the same expansion.

A detailed explanation of how a hyperasymptotic expansion of the form \eqref{ahyper} can be used to compute the connection coefficients $K_{\ell,j}$ numerically is given in \cite{OldeDaalhuis1998b}. We do not reproduce the procedure here.

\section{Error bounds for the inverse factorial series}\label{sec5}

In this section, we derive explicit error bounds for the inverse factorial expansions \eqref{winvfact}, under the additional assumption that $\Re(\mu_j) < 1$ for each $1\le j\le n$. To this end, for each $1\le j\le n$ and admissible direction $\eta$, we introduce the function $\Delta_{\lambda_j}y_j(t,\eta)$ defined by
\[
2\pi \im\, \Delta_{\lambda_j}y_j(t,\eta)=
y_j \big( t + \lambda_j  ,\eta \big) - y_j \big( \big( t + \lambda_j  \big)\e^{ - 2\pi \im} ,\eta  \big).
\]
Proposition \ref{prop1} immediately implies that this function is analytic on the shifted plane $\mathcal{P}_\eta - \lambda_j$, with a convergent expansion of the form
\begin{equation}\label{Deltay}
\Delta_{\lambda_j}y_j(t,\eta)= - \sum_{s = 0}^\infty \frac{a_{s,j}}{\Gamma \left( s - \mu _j  + 1 \right)}\left( \frac{t}{\lambda _j}\right)^{s - \mu _j },\quad 
\left| t \right| < \min_{\ell  \ne j} \left| \lambda_{\ell ,j} \right|,
\end{equation}
and it admits analytic continuation along any path avoiding the points $0$ and $\lambda_{\ell,j}$ for $\ell \neq j$. Furthermore, if $S$ is a sector in the complex $t$-plane of the form \eqref{Ssec} with $R > \max_\ell |\lambda_{\ell,j}|$, then for any fixed analytic continuation of $\Delta_{\lambda_j}y_j(t,\eta)$ throughout $S$, we have
\[
\lim_{t \to \infty} t^{-(a + \varepsilon)} \Delta_{\lambda_j}y_j(t,\eta) = 0, \quad t \in S,
\]
for every $\varepsilon > 0$, where $a$ denotes the real part of the root of $f_0(z) = 0$ with the largest real part.

We first present a bound for the remainder term in the asymptotic expansion \eqref{Ylarge} of the associated function $Y_j(u, \eta)$. This result will subsequently be applied to derive error bounds for the inverse factorial expansions \eqref{winvfact} of $w_j(z, \eta)$, and may also be of independent interest. The proof will be given later in this section.

\begin{proposition}\label{prop4}
Let $\eta \in \mathbb{R}$ be an admissible direction. Fix the phase of $\lambda_{\ell,j} = \lambda_\ell - \lambda_j$ by setting $\arg \lambda_{\ell,j}=\theta_{j,\ell} \in (\eta-2\pi,\eta)$ and assume that $\theta_{j,\ell_1} \neq \theta_{j,\ell_2}$ whenever $\ell_1 \neq \ell_2$ and $K_{\ell_1,j}, K_{\ell_2,j} \neq 0$.
For each $1 \le j \le n$ and non-negative integer $N$, let $\mathsf{R}_j(u, \eta; N)$ denote the remainder obtained by truncating the asymptotic expansion \eqref{Ylarge} after $N$ terms:
\begin{equation}\label{Yexpansion}
Y_j(u, \eta) = \e^{\lambda_j u} u^{-1}  \sum_{s=0}^{N-1} a_{s,j} \left( \lambda_j u \e^{-\pi \im} \right)^{\mu_j - s} + \mathsf{R}_j (u, \eta;N).
\end{equation}
Then
\begin{gather}\label{Rbound}
\begin{split}
\left|\mathsf{R}_j(u, \eta; N)\right| \le 
\left| \e^{\lambda_j u} u^{-1} \right|
\sum_{\ell \ne j} & \left| K_{\ell,j} \right| \left| \left(\lambda_{\ell,j}u\right)^{ \mu_j-N} \right| 
\int_{0}^{+\infty}
\frac{
\left| \Delta_{\lambda_\ell} y_\ell\big(\lambda_{\ell,j} t, \eta\big) \right|
}{
\left(1 + t\right)^{N - \Re(\mu_j) + 1}
}
\id t \\ & \times
\sup_{r > 0} r
\left|
F^{(1)}\!\left(
r \e^{\im(\theta_{j,\ell} + \arg u)};
\mytop{N - \mu_j + 1}{1}
\right)
\right|,
\end{split}
\end{gather}
provided $u \in \mathscr{S}(\eta)$, $N-\Re(\mu_j) > \max(0,a)$, and $\Re(\mu_\ell) < 1$ for all $\ell$ such that $K_{\ell,j}\neq 0$.
\end{proposition}

\begin{remarks}
\noindent
\begin{enumerate}[label=(\roman*), leftmargin=*, labelindent=0pt,itemsep=0.3em]
    \item Note that Proposition \ref{prop4} provides error bounds for the asymptotic expansions of Laplace-type integrals of the form \eqref{Yint}, regardless of whether the Borel transform $y_j(t,\eta)$ arises from a difference equation, provided it satisfies the conditions stated in Theorem \ref{thm1}. The bound \eqref{Rbound} is analogous to those established for integrals with saddle points \cite[\S5]{Bennett2018}.
    \item The bound \eqref{Rbound} can be further simplified by using estimates for the level-one hyperterminant.
For example, if $\mu_j$ is real, then
\begin{gather}\label{F1bound}
\begin{split}
&\sup_{r > 0} r
\left|
 F^{(1)}\!\left(
r \e^{\im(\theta_{j,\ell} + \arg u)};
\mytop{N - \mu_j + 1}{1}
\right)
\right| \le \Gamma \left( N - \mu _j  + 1 \right) \\ & \quad \times
\begin{cases}
1, & \text{if }\; \left| \theta_{j,\ell} + \arg u \right| \le \frac{\pi}{2}, \\[6pt]
\min\left( \left| \csc\!\left( \theta_{j,\ell} + \arg u \right) \right|, \chi \left( N - \mu_j + 1 \right) + 1 \right),
& \text{if }\; \frac{\pi}{2} < \left| \theta_{j,\ell} + \arg u \right| \le \pi, \\[6pt]
\dfrac{\sqrt{ 2\pi \left( N - \mu_j + 1 \right) }}{ \left| \cos \left( \theta_{j,\ell} + \arg u \right) \right|^{N - \mu_j + 1} }
+ \chi \left( N - \mu_j + 1 \right) + 1,
& \text{if }\; \pi < \left| \theta_{j,\ell} + \arg u \right| < \frac{3\pi}{2},
\end{cases}
\end{split}
\end{gather}
where $\chi (p) = \frac{\sqrt \pi  \Gamma \left( \frac{p}{2} + 1 \right)}{\Gamma \left( \frac{p + 1}{2} \right)}$ (cf. \cite[Appendix~B]{Bennett2018}). For general complex values of $\mu_j$, the methods used to bound a related function in \cite[Appendix~A]{Nemes2017} could be used.
\item From formula \eqref{Rexplicit} below, we have
\begin{align*}
&\e^{\lambda _j u} u^{ - 1} a_{N,j} \left( \lambda _j u\e^{ - \pi \im} \right)^{\mu _j  - N} =\mathsf{R}_j (u,\eta ;N) - \mathsf{R}_j (u,\eta ;N + 1) \\ & \quad = -
 \e^{\lambda _j u} u^{ - 1}  \Gamma \left( N - \mu _j  + 1 \right) \sum_{\ell \ne j} 
K_{\ell,j} \left( \lambda _{\ell ,j} u\e^{ - \pi \im}  \right)^{\mu _j  - N}
\int_{0}^{+\infty}
\frac{
\Delta_{\lambda_\ell} y_\ell\big(\lambda_{\ell,j} t , \eta\big)
}{
\left(1 + t\right)^{N - \mu_j + 1}
} \id t.
\end{align*}
Thus, the right-hand side of the bound \eqref{Rbound} is seen to be closely related to the absolute value of the first omitted term in the asymptotic expansion \eqref{Ylarge}.
\item The assumption in Proposition \ref{prop4} that $\theta_{j,\ell_1} \neq \theta_{j,\ell_2}$ whenever $\ell_1 \neq \ell_2$ and $K_{\ell_1,j}, K_{\ell_2,j} \neq 0$ ensures that no singularities of $\Delta_{\lambda_j}y_j(t,\eta)$ lie on the integration path in \eqref{Rbound}. This restriction can be relaxed by deforming the contour to pass to the left or right of any such singularities.
\end{enumerate}
\end{remarks}

We now turn to the error bounds for the inverse factorial expansions \eqref{winvfact} of the solutions $w_j (z,\eta)$.

\begin{theorem}\label{thm2}
Let $\eta \in \mathbb{R}$ be an admissible direction such that $|\eta - \arg \lambda_j| < \pi$ and $-\frac{\pi}{2}+\eta^- < \arg \lambda_j < \frac{\pi}{2}+\eta^+$. Fix the phase of $\lambda_{\ell,j} = \lambda_\ell - \lambda_j$ by setting $\arg \lambda_{\ell,j}=\theta_{j,\ell} \in (\eta-2\pi,\eta)$ and assume that $\theta_{j,\ell_1} \neq \theta_{j,\ell_2}$ whenever $\ell_1 \neq \ell_2$ and $K_{\ell_1,j}, K_{\ell_2,j} \neq 0$. For each $1 \le j \le n$ and non-negative integer $N$, let $R_j (z,\eta ;N)$ denote the remainder obtained by truncating the inverse factorial expansion \eqref{winvfact} after $N$ terms:
\[
w_j (z,\eta)= \lambda_j^{-z} \sum_{s = 0}^{N-1}  a_{s,j} \Gamma \left(z + \mu _j  - s\right)+R_j (z,\eta ;N).
\]
Then
\begin{gather}\label{invfactRbound}
\begin{split}
\left| R_j (z,\eta ;N) \right| \le \left| \lambda _j^{ - z}  \right|\Gamma \left( \Re \left( z + \mu _j \right) - N \right)\sum_{\ell  \ne j} & \left| K_{\ell ,j}  \right|\left| \left( \frac{\lambda _{\ell ,j} }{\lambda _j }\e^{\pi \im} \right)^{\mu _j  - N} \right| 
\int_{0}^{+\infty}
\frac{
\left| \Delta_{\lambda_\ell} y_\ell\big(\lambda_{\ell,j} t, \eta\big) \right|
}{
\left(1 + t\right)^{N - \Re(\mu_j) + 1}
}
\id t \\ & \times
\sup_{r > 0} r
\left|
F^{(1)}\!\left(
r \e^{\im(\pi+\theta_{j,\ell}- \arg \lambda_j)};
\mytop{N - \mu_j + 1}{1}
\right)
\right|,
\end{split}
\end{gather}
provided $\Re(z)>N-\Re(\mu_j) > \max(0,a)$, and $\Re(\mu_\ell) < 1$ for all $\ell$ such that $K_{\ell,j}\neq 0$.
\end{theorem}

The proof of Proposition \ref{prop4} is given first, followed by the proof of Theorem \ref{thm2}.

\begin{proof}[Proof of Proposition \ref{prop4}] Let $C$ be a closed contour encircling $t$ and $\lambda_j$ in the positive direction, with all other points $\lambda_\ell$ lying outside $C$. Then, by the complex Taylor theorem \cite[\S3.1]{Ahlfors1966}, we have
\[
y_j (t,\eta) = \sum_{s = 0}^{N-1}  a_{s,j} \Gamma \left(\mu _j-s\right)\left( 1 - \frac{t}{\lambda _j } \right)^{s - \mu _j }+\frac{1}{2\pi \im} \sum_{\ell \ne j}\oint_{C} 
\frac{1}{\tau - t} 
\left( \frac{t - \lambda_j}{\tau - \lambda_j} \right)^{N - \mu_j} 
y_j(\tau, \eta)\id\tau,
\]
for any non-negative integer $N$. We now introduce branch cuts in the $t$-plane, each starting at $\lambda_\ell$ ($\ell \ne j$) and extending to infinity along the ray $\arg\left(t-\lambda_\ell\right) =\theta_{j,\ell }$. Assuming that $N > \Re(\mu _j ) + a$, we can deform the contour $C$ so that it encircles each of these branch cuts, with the arcs at infinity making no contribution. This yields
\[
y_j (t,\eta) = \sum_{s = 0}^{N-1}  a_{s,j} \Gamma \left(\mu _j-s\right)\left( 1 - \frac{t}{\lambda _j } \right)^{s - \mu _j }-\frac{1}{2\pi \im} \sum_{\ell \ne j}\int_{\gamma_\ell(\theta_{j,\ell })} 
\frac{1}{\tau - t} 
\left( \frac{t - \lambda_j}{\tau - \lambda_j} \right)^{N - \mu_j} 
y_j(\tau, \eta)\id\tau,
\]
where the contour $\gamma_\ell(\theta_{j,\ell})$ is defined in the same manner as the contours appearing in \eqref{R0integral}. Substituting into \eqref{Yint} yields \eqref{Yexpansion} with
\begin{equation}\label{Rformula}
\mathsf{R}_j(u, \eta; N) = 
- \frac{1}{2\pi \im} \sum_{\ell \ne j} 
\int_{\gamma_j(\eta)} 
\frac{1}{2\pi \im} \int_{\gamma_\ell(\theta_{j,\ell })} 
\frac{\e^{u t}}{\tau - t} 
\left( \frac{t - \lambda_j}{\tau - \lambda_j} \right)^{N - \mu_j} 
y_j(\tau,\eta)\id\tau\id t,
\end{equation}
valid for any $u \in \mathscr{S}(\eta)$. Next, we perform that change of variable from $t$ and $\tau$ to $v$ and $\tau$ via
\[
t = \frac{v}{u}\left( \tau  - \lambda _j \right) + \lambda _j .
\]
Under this substitution, \eqref{Rformula} becomes
\[
\mathsf{R}_j(u, \eta; N) = 
- \e^{\lambda_j u} u^{\mu_j-N} \frac{1}{2\pi \im} 
\sum_{\ell \ne j} 
\int_{\gamma_0(\arg u  +\eta - \theta_{j,\ell } )} 
\frac{\e^{-\lambda_j v} v^{N - \mu_j}}{u- v}  \frac{1}{2\pi \im} 
\int_{\gamma_\ell(\theta_{j,\ell })} 
\e^{v\tau} y_j(\tau, \eta)\id\tau\id v,
\]
where $\gamma_0(\arg u + \eta - \theta_{j,\ell})$ starts at infinity along the left-hand side of the branch cut $\arg v = \arg u + \eta - \theta_{j,\ell}$, encircles the origin counter-clockwise, and returns to infinity along the right-hand side of the same cut. Along the left-hand side, $\arg v = \arg u + \eta - \theta_{j,\ell} - 2\pi$, while along the right-hand side $\arg v = \arg u + \eta - \theta_{j,\ell}$. If $N > \Re(\mu_j) - 1$, the contour may be collapsed onto the two sides of the branch cut, yielding
\[
\mathsf{R}_j(u, \eta; N)
= \e^{\lambda_j u} u^{\mu_j -N}
\sum_{\ell \ne j} 
K_{\ell, j} \int_0^{[\arg u  + \eta - \theta_{j,\ell } ]} 
\frac{\e^{-\lambda_j v} v^{N - \mu_j}}{u - v} \frac{1}{2\pi \im} 
\int_{\gamma_\ell(\theta_{j,\ell })} 
\e^{v\tau} y_\ell(\tau, \eta)\id\tau\id v,
\]
where \eqref{connection} has been used. Assuming $\Re(\mu_\ell) < 1$ for all $\ell$ with $K_{\ell,j} \neq 0$, we collapse the contours $\gamma_\ell(\theta_{j,\ell})$ onto the two sides of the corresponding branch cuts to obtain
\begin{align*}
\mathsf{R}_j(u, \eta; N)
&= \e^{\lambda_j u} u^{\mu_j-N}
\sum_{\ell \ne j} 
K_{\ell,j}
\int_{0}^{[\arg u + \eta - \theta_{j,\ell}]}
\frac{
\e^{\lambda_{\ell,j} v} v^{N - \mu_j}
}{u - v}
\int_{0}^{[\theta_{j,\ell}]}
\e^{v\tau} \Delta_{\lambda_\ell} y_\ell(\tau, \eta) \id \tau \id v\\
&=\e^{\lambda_j u} u^{\mu_j-N}  
\sum_{\ell \ne j} 
K_{\ell,j} \lambda_{\ell,j} 
\int_{0}^{[\arg u + \eta - \theta_{j,\ell}]}
\frac{
\e^{\lambda_{\ell,j} v} v^{N - \mu_j}
}{u - v}  \\ & \hspace{12em} \times \int_{0}^{+\infty}
\e^{\lambda_{\ell,j} v t} \Delta_{\lambda_\ell} y_\ell\big(\lambda_{\ell,j} t , \eta\big) \id t \id v.
\end{align*}
Finally, interchanging the order of integration and applying an appropriate change of variables together with the definition of the level-one hyperterminant, we obtain
\begin{gather}\label{Rexplicit}
\begin{split}
\mathsf{R}_j(u, \eta; N)
=\e^{\lambda_j u} u^{-1}
\sum_{\ell \ne j} &
K_{\ell,j}\left(\lambda_{\ell,j}u\right)^{\mu_j-N}
\int_{0}^{+\infty}
\frac{
\Delta_{\lambda_\ell} y_\ell\big(\lambda_{\ell,j} t , \eta\big)
}{
\left(1 + t\right)^{N - \mu_j + 1}
}\\ & \times
\lambda_{\ell,j} (1 + t) u
F^{(1)}\!\left(\lambda_{\ell,j} (1 + t) u;
\mytop{N - \mu_j + 1}{1}
\right) \id t,
\end{split}
\end{gather}
provided $N-\Re(\mu_j) > 0$. Applying a straightforward estimate to the right-hand side then produces the required bound \eqref{Rbound}.
\end{proof}

\begin{proof}[Proof of Theorem \ref{thm2}]
By choosing an angle $\widetilde{\eta} \in \mathcal{I}_\eta$ with $\left| \widetilde{\eta} - \arg \lambda_j \right| < \frac{\pi}{2}$, we can write
\[
R_j (z,\eta ;N) = \e^{ - \pi \im z} \int_0^{\left[ \pi  - \arg \lambda _j \right]} u^z \mathsf{R}_j \big(u,\widetilde \eta ;N\big)\id u 
\]
(cf.~\eqref{wremainder} and \eqref{remainderformula}). Applying the triangle inequality together with the bound \eqref{Rbound} (with $\widetilde{\eta}$ substituted for $\eta$) yields \eqref{invfactRbound} with $\Delta_{\lambda_\ell} y_\ell\big(\lambda_{\ell,j} t, \widetilde{\eta}\big)$ replacing $\Delta_{\lambda_\ell} y_\ell\big(\lambda_{\ell,j} t, \eta\big)$. However, since $\widetilde{\eta} \in \mathcal{I}_\eta$, these two functions coincide by the principle of analytic continuation.
\end{proof}

\section{Applications}\label{sec6}
In this section, we present two examples that demonstrate the applicability of our results. In the first, we derive a new expansion for the Gauss hypergeometric function with a large third parameter, accompanied by explicit error bounds. The second concerns a third-order difference equation, for which we obtain hyperasymptotic expansions and show how the connection coefficients in these expansions can be computed numerically from the behaviour of the late coefficients.

\subsection{Gauss hypergeometric function with a large third parameter} From the contiguous relation for the Gauss hypergeometric function \cite[Eq.~\href{http://dlmf.nist.gov/15.5.E18}{(15.5.18)}]{DLMF}, it follows that
\[
w(\lambda) = \frac{\Gamma (c- a+\lambda)\Gamma (c- b+\lambda )}{\Gamma (c+\lambda)}{}_2F_1 \!\left(\mytop{a,b}{c+\lambda};z\right)
\]
satisfies the second-order difference equation
\begin{equation} \label{example1}
w(\lambda + 2) + f_1 (\lambda)w(\lambda + 1) + f_0 (\lambda)w(\lambda) = 0,
\end{equation}
where
\begin{align*}
f_0 (\lambda) & = \frac{z - 1}{z}\lambda (\lambda  - 1) + \frac{z - 1}{z}(2c - a - b + 1)\lambda  + \frac{z - 1}{z}(c - a)(c - b),\\
f_1 (\lambda) & = \frac{1 - 2z}{z}\lambda+\frac{1 - 2z}{z}c   + a + b - 1.
\end{align*}
The characteristic roots $\lambda_j$ and the corresponding constants $\mu_j$ are
\[
\lambda_1=1,\quad \lambda _2  = \frac{z}{z - 1}, \quad \mu _1  =  c- a-b, \quad \mu _2  =  c- 1.
\]
Throughout, we assume that neither $c$ nor $c - a - b$ is an integer. We also temporarily restrict $z$ so that $|\arg\left(1-z\right)|<\pi$ and $\Im(z)\ne 0$. With the initial values $a_{0,j} = 1$ for $j = 1,2$, the recurrence \eqref{arec} can be solved in closed form, yielding
\[
a_{s,1} = \frac{(a)_s (b)_s }{s!}(z - 1)^s ,\quad a_{s,2} = (-1)^s\frac{(1 - a)_s (1 - b)_s }{s!}z^s,
\]
valid for $s \ge 0$.

We recall from Proposition \ref{prop3} that the exact solutions constructed in this paper depend not only on $z$ but also on a real parameter $\eta$, assumed to be admissible in the sense of Section \ref{sec2}. Under the present constraints on $z$, the choice $\eta = 0$ is admissible, and we adopt it here. By Proposition \ref{prop3}, the difference equation \eqref{example1} has, for $\Re(\lambda) > \max(\Re(a-c), \Re(b-c))$, two solutions $w_1(\lambda,0)$ and $w_2(\lambda,0)$ whose asymptotic behaviour, as $\Re(\lambda) \to +\infty$ with $\Im(\lambda) = \mathcal{O}\big(\sqrt{\Re(\lambda)}\big)$, is given by
\[
w_1 (\lambda,0) \sim \sum_{s = 0}^\infty  \frac{(a)_s (b)_s }{s!}(z - 1)^s \Gamma (\lambda  + c - a - b - s),
\]
and
\[
w_2 (\lambda,0) \sim \left( \frac{z - 1}{z} \right)^\lambda  \sum_{s = 0}^\infty(-1)^s  \frac{(1 - a)_s (1 - b)_s }{s!}z^s \Gamma (\lambda  + c - 1 - s).
\]
Now there exist periodic functions $C_1(\lambda)$ and $C_2(\lambda)$ of $\lambda$ with period $1$ such that
\begin{equation}\label{Cexpression}
\frac{\Gamma (c- a+\lambda)\Gamma (c- b+\lambda )}{\Gamma (c+\lambda)}{}_2F_1 \!\left(\mytop{a,b}{c+\lambda};z\right) =C_1(\lambda)w_1(\lambda,0)+C_2(\lambda)w_2(\lambda,0).
\end{equation}
To determine these functions explicitly, we first observe that the known large-$\lambda$ asymptotic power series \cite[Eq.~\href{http://dlmf.nist.gov/15.12.E3}{(15.12.3)}]{DLMF} for the Gauss hypergeometric function yields
\[
\frac{\Gamma (c- a+\lambda)\Gamma (c- b+\lambda )}{\Gamma (c+\lambda)}{}_2F_1 \!\left(\mytop{a,b}{c+\lambda};z\right) \sim \frac{\Gamma (c - a + \lambda )}{\lambda ^b } \sim \Gamma (\lambda  + c - a - b),
\]
as $\lambda\to+\infty$. Hence, dividing both sides of \eqref{Cexpression} by $\Gamma(\lambda + c - a - b)$ and letting $\lambda \to +\infty$, we find
\begin{align*}
1 & = \lim_{\lambda \to +\infty} C_1(\lambda)
  + \lim_{\lambda \to +\infty} C_2(\lambda)
    \lim_{\lambda \to +\infty} 
      \left( \frac{z - 1}{z} \right)^{\lambda} 
      \frac{\Gamma(\lambda + c - 1)}{\Gamma(\lambda + c - a - b)}
\\ & = \lim_{\lambda  \to  + \infty } C_1 (\lambda ) + \lim_{\lambda  \to  + \infty } C_2 (\lambda )\lim_{\lambda  \to  + \infty } \left( \frac{z - 1}{z}\right)^\lambda  \lambda ^{a + b - 1} .
\end{align*}
If $\Re(z) < \frac{1}{2}$, then $\left| \frac{z - 1}{z} \right| > 1$, and since both $C_1(\lambda)$ and $C_2(\lambda)$ are periodic, it follows that $C_1(\lambda) = 1$ and $C_2(\lambda) = 0$. Therefore,
\begin{equation}\label{Fexpression}
\frac{\Gamma (c- a+\lambda)\Gamma (c- b+\lambda )}{\Gamma (c+\lambda)}{}_2F_1 \!\left(\mytop{a,b}{c+\lambda};z\right)=w_1(\lambda,0),
\end{equation}
provided that $\Re(\lambda) > \max(\Re(a-c), \Re(b-c))$, $\Re(z) < \frac{1}{2}$, and $\Im(z)\ne 0$.

We can extend the $z$-domain of validity of \eqref{Fexpression} and the $z$-domain of definition of $w_1(\lambda,0)$ to $|\arg\left(1 - z\right)| < \pi$ by computing the corresponding Borel transform $y_1(t,0)$. By definition,
\begin{align*}
y_1(t,0) & = \sum_{s=0}^\infty \frac{(a)_s (b)_s}{s!} (z - 1)^s \Gamma(c - a - b - s) (1 - t)^{s + a + b - c} \\ &
= \Gamma(c - a - b) (1 - t)^{a + b - c}  {}_2F_1\!\left(
\mytop{a, b}{a + b - c + 1}; (1 - z)(1 - t)
\right),
\end{align*}
valid when $|(1 - z)(1 - t)|<1$. By analytic continuation, this condition can be relaxed to $|\arg\left(1 - z\right)| < \pi$ and $|\arg\left(1 - t\right)| < \pi$. Therefore, $w_1(\lambda, 0)$ is analytic in $z$ for $|\arg\left(1 - z\right)| < \pi$, and by analytic continuation, the equality \eqref{Fexpression} holds throughout this domain.

Thus, we have established the inverse factorial expansion
\begin{equation}\label{hypergeoexp}
\frac{\Gamma (c- a+\lambda)\Gamma (c- b+\lambda )}{\Gamma (c+\lambda)}{}_2F_1 \!\left(\mytop{a,b}{c+\lambda};z\right)\sim \sum_{s = 0}^\infty  \frac{(a)_s (b)_s }{s!}(z - 1)^s \Gamma (\lambda  + c - a - b - s),
\end{equation}
as $\Re(\lambda) \to +\infty$ with $\Im(\lambda) = \mathcal{O}\big(\sqrt{\Re(\lambda)}\big)$,
under the conditions that $|\arg\left(1-z\right)| < \pi$ and that neither $c$ nor $c-a-b$ is an integer. To the best of our knowledge, the inverse-factorial expansion \eqref{hypergeoexp} appears to be new; it differs from the standard large-parameter expansions for ${}_2F_1$ available in the literature (cf.~\cite[Eqs.~\href{http://dlmf.nist.gov/15.12.E2}{(15.12.2)} and \href{http://dlmf.nist.gov/15.12.E3}{(15.12.3)}]{DLMF} and \cite[Eq. (C.4)]{Bissi2025}).

We conclude this section by deriving error bounds for the expansion \eqref{hypergeoexp}. Throughout, we continue to assume that neither $c$ nor $c-a-b$ is an integer. We begin by determining the connection coefficient $K_{2,1}$. Applying \eqref{ahyper} with $\ell = 1$ gives
\[
\frac{(a)_N (b)_N }{N!}(z-1)^N  \sim K_{2,1} z^{c - 1} (z - 1)^{N+a + b - c} \sum_{s = 0}^\infty  (-1)^s\frac{(1 - a)_s (1 - b)_s }{s!} \Gamma (N + a + b - 1 - s)
\]
as $N\to+\infty$. Comparing this with the known inverse factorial expansion for a ratio of gamma functions \cite[Eq.~\href{http://dlmf.nist.gov/5.11.E19}{(5.11.19)}]{DLMF}, we find
\[
K_{2,1}  = \frac{z^{1 - c} (z - 1)^{c - a - b} }{\Gamma (a)\Gamma (b)}.
\]
Furthermore, from \eqref{Deltay} and analytic continuation, we have
\[
\Delta _{\lambda _2 } y_2 \big( \lambda _{2,1} t,0 \big) = - \frac{1}{\Gamma (2 - c)}\left( \frac{t}{z}\right)^{1 - c} {}_2F_1\! \left(   \mytop{1 - a,1 - b}{2 - c} ; - t \right),
\]
valid for any $t > 0$. Thus, writing
\[
\frac{\Gamma (c- a+\lambda)\Gamma (c- b+\lambda )}{\Gamma (c+\lambda)}{}_2F_1 \!\left(\mytop{a,b}{c+\lambda};z\right)= \sum_{s = 0}^{N-1} \frac{(a)_s (b)_s }{s!}(z - 1)^s \Gamma (\lambda  + c - a - b - s)+R_1 (\lambda,0 ;N),
\]
Theorem~\ref{thm2} gives the bound
\begin{gather}\label{2F1bound}
\begin{split}
\left|R_1 (\lambda,0 ;N)\right| & \le \frac{|z - 1|^{N} \left| \e^{\pi \im (c - a - b)} \right|}{| \Gamma(a) \Gamma(b) \Gamma(2 - c) |} \Gamma( \Re (\lambda + c - a - b) - N  ) \\ & \quad \times \int_{0}^{+\infty} \frac{t^{1 - \Re(c)}}{\left(1 + t\right)^{N + \Re(a + b - c) + 1}} \left| {}_2 F_1\!\left( \mytop{1-a,1-b}{2-c} ; -t \right) \right| \id t \\ & \quad\times \sup_{r > 0} r \left| F^{(1)}\!\left( r \e^{-\im \arg(1 - z)} ; \mytop{N + a + b - c + 1}{1} \right) \right|,
\end{split}
\end{gather}
valid whenever $\Re(\lambda + c - a - b)>N > \max(-\Re(a), -\Re(b))$, $\Re(c) < 2$, $|\arg\left(1-z\right)| < \pi$ and $\Im(z)\ne 0$. The requirement that $\Im(z) \ne 0$ can be relaxed using a continuity argument. We now show that a more explicit estimate for $R_1 (\lambda,0 ;N)$ can be obtained if the condition $\Re(c) < 2$ is replaced by the stronger requirement $\Re(c) < \Re(b) + 1 < 2$.
By \cite[Eq.~\href{http://dlmf.nist.gov/15.6.E1}{(15.6.1)}]{DLMF},
\begin{align*}
& \left| {}_2F_1\!\left(
\mytop{1 - a, 1 - b}{2-c}; -t
\right) \right|
 =
\left| \frac{\Gamma(2 - c)}{\Gamma(1 - b)\Gamma(b - c + 1)}
\int_0^1 \frac{\tau^{-b} (1 - \tau)^{b - c}}{\left(1 + t \tau\right)^{1 - a}} \id\tau \right|
\\ & \quad \le
\frac{|\Gamma(2 - c)| \Gamma(1 - \Re(b)) \Gamma\left(\Re (b - c ) + 1\right)}{\Gamma(2 - \Re(c)) |\Gamma(1 - b) \Gamma(b - c + 1)|}
{}_2F_1\!\left(
\mytop{1 - \Re(a), 1 - \Re(b)}{2-\Re(c)}; -t
\right),
\end{align*}
valid when $\Re(c) < \Re(b) + 1 < 2$.
Furthermore, by \cite[Eq.~(21.1.1.16)]{Prudnikov1990},
\begin{multline*}
 \int_0^{+\infty} \frac{t^{1 - \Re(c)}}{\left(1 + t\right)^{N + \Re(a + b - c) + 1}} {}_2F_1\!\left(\mytop{1 - \Re(a), 1 - \Re(b)}{2 - \Re(c)}; -t \right) \id t
\\ = \frac{\Gamma(2 - \Re(c)) \Gamma(N + \Re(a)) \Gamma(N + \Re(b))}{N! \Gamma(N + \Re(a + b - c) + 1)},
\end{multline*}
provided $N > \max(-\Re(a), -\Re(b))$ and $\Re(c) < 2$. Combining these results with \eqref{2F1bound} gives the following estimate for $R_1 (\lambda,0 ;N)$:
\begin{align*}
\left|R_1 (\lambda,0 ;N)\right| & \le
\left|\e^{\pi \im(c - a - b)}\right| \frac{|\sin(\pi b)|}{|\sin(\pi \Re(b))|} \frac{\Gamma(\Re(a)) \Gamma(\Re(b - c) + 1)}{|\Gamma(a) \Gamma(b - c + 1)|} \\ & \quad\times\frac{(\Re(a))_N (\Re(b))_N}{N!} |z - 1|^N \Gamma(\Re(\lambda + c - a - b) - N) \\ & \quad \times\frac{1}{\Gamma(N + \Re(a + b - c) + 1)} \sup_{r > 0} r \left| F^{(1)}\left(r \e^{-\im \arg(1 - z)} ; \mytop{N + a + b - c + 1}{1} \right) \right|,
\end{align*}
valid when $\Re(\lambda + c - a - b)>N > \max(-\Re(a), -\Re(b))$ and $\Re(c) < \Re(b) + 1 < 2$. If $a$, $b$, and $c$ are all real, this bound can be further simplified using \eqref{F1bound}:
\begin{align*}
\left|R_1 (\lambda,0 ;N)\right| & \le \frac{(a)_N (b)_N}{N!} |z - 1|^N \Gamma(\Re(\lambda) + c - a - b - N) 
\\ & \quad\times 
\begin{cases}
1, & \text{if }\; |\arg\left(1 - z\right)| \le \tfrac{\pi}{2}, \\[6pt]
\min\left(|\csc(\arg\left(1 - z\right))|, \chi(N + a + b - c + 1) + 1\right), & \text{if }\; \tfrac{\pi}{2} < |\arg\left(1 - z\right)| < \pi,
\end{cases}
\end{align*}
provided that $\Re(\lambda) + c - a - b>N>\max(-a,-b)$ and $c<b+1<2$.

\begin{remarks}
\noindent
\begin{enumerate}[label=(\roman*), leftmargin=*, labelindent=0pt,itemsep=0.3em]
\item The second solution, $w_2(\lambda,0)$, can likewise be expressed in terms of the Gauss hypergeometric function. In particular,
\[
w_2(\lambda, 0) =
\left( \frac{z-1}{z} \right)^\lambda
\frac{\Gamma(c-a+\lambda) \Gamma(c-b+\lambda)}{\Gamma(c-a-b+1+\lambda)}
{}_2F_1\!\left( \mytop{1-a, 1-b}{c-a-b+1+\lambda}; 1-z \right),
\]
provided that neither $c$ nor $c-a-b$ is an integer, $\Re(\lambda) > \max(\Re(a-c),\Re(b-c))$, $|\arg\left(1-z\right)| < \pi$, and $|\arg z| < \pi$.
\item It is possible to derive hyperasymptotic refinements of \eqref{hypergeoexp} using Theorem \ref{thm1}. However, even the level-one re-expansion involves hypergeometric functions with a large parameter. As a result, the approximants are essentially as complicated as the original function we aim to approximate. Therefore, the hyperasymptotic expansions in this setting offer little practical advantage.
\item The problem of determining the large-$\nu$ asymptotics of the associated Legendre functions $P_\nu^\mu(z)$ and $Q_\nu^\mu(z)$ may be reformulated in terms of the Gauss hypergeometric function with a large third parameter. Applying the results of this subsection then recovers some of the results on the associated Legendre functions given in \cite{Nemes2020}.
\end{enumerate}
\end{remarks}

\subsection{A third-order difference equation} As a second example, consider the third-order difference equation
\begin{equation}\label{example2}
w(z + 3) + f_2 (z)w(z + 2) + f_1 (z)w(z + 1) + f_0 (z)w(z) = 0,
\end{equation}
where
\begin{align*}
f_0 (z) & =  - \tfrac{1}{2}(z - 2)(z - 1)z - \tfrac{15}{4}(z - 1)z,
\\ f_1 (z) & = (z - 1)z + 5z,\quad f_2 (z) =  - \tfrac{1}{2}z - \tfrac{5}{4}.
\end{align*}
The characteristic roots $\lambda_j$ and the associated constants $\mu_j$ are given by
\[
\lambda _1  = 2,\quad \lambda _2  = \im,\quad \lambda _3  =  - \im,\quad \mu _1  = \mu _2  = \mu _3  = \tfrac{1}{2}.
\]
The coefficients $a_{s,j}$ can be computed recursively from \eqref{arec}, starting with the initial values $a_{0,j} = 1$. In particular, $a_{s,1}$ are real, and the remaining two sets of coefficients satisfy the conjugacy relation $a_{s,3} = \overline{a_{s,2}}$.

We recall that the exact solutions constructed in Proposition~\ref{prop3} are functions depending on both $z$ and an admissible direction parameter $\eta$ (see Section \ref{sec2}). We choose $\eta = 0$ as the admissible direction. Our goal is to compute $w_{1}(z,0)$ at $z = 30 + \im$ using level-zero, level-one, and level-two hyperasymptotic approximations. To this end, we first determine the connection coefficients $K_{\ell,j}$ defined in \eqref{connection}.

To compute $K_{2,1}$ and $K_{3,1}$, we apply \eqref{ahyper} with $j=1$ and $\ell=1$. Specifically, if
\[
N_{2,3}^{(1)}  \approx \bigg( 1 - \frac{\alpha _1^{(0)} }{\alpha _1^{(1)} } \bigg)N_1^{(0)}  = \frac{2}{\sqrt 5  + 2}N_1^{(0)} ,
\]
then we have
\begin{gather}\label{latecoeff1}
\begin{split}
a_{N_1^{(0)} ,1} & = \frac{1 + \im}{2}K_{2,1} \left( \frac{2}{\im - 2} \right)^{N_1^{(0)}} \sum_{s = 0}^{N_2^{(1)}  - 1} a_{s,2} \left( \frac{\im}{\im - 2} \right)^{ - s} \Gamma \big( N_1^{(0)}  - s \big) \\ &\qquad + \frac{1 - \im}{2}K_{3,1} \left( \frac{2}{ - \im - 2} \right)^{N_1^{(0)} } \sum_{s = 0}^{N_3^{(1)}  - 1} a_{s,3} \left( \frac{\im}{\im + 2} \right)^{ - s} \Gamma \big( N_1^{(0)}  - s \big)
\\ &\qquad + \left( \frac{2}{\sqrt 5  + 2} \right)^{N_1^{(0)} } \Gamma \big( N_1^{(0)}  + \tfrac{3}{2} \big)\mathcal{O}(1),
\end{split}
\end{gather}
as $N_1^{(0)}\to+\infty$. From the recurrence, we calculate
\begin{gather}\label{coeffnum}
\begin{split}
a_{100,1}&=\phantom{-}1.1142187816307847845\ldots\times 10^{151},\\
a_{101,1}&=-1.0529263779207865718\ldots\times 10^{153}.
\end{split}
\end{gather}
Next, neglecting the remainder term in \eqref{latecoeff1}, we evaluate the right-hand side for $N_1^{(0)}=100,101$, $N_2^{(1)}=N_3^{(1)}=47$, and equate the results to the corresponding values in \eqref{coeffnum}. This leads to a linear system of two equations. Solving this system yields
\begin{align*}
K_{2,1}  &=  - 0.54527032667963220005 + 0.23807964635130112660\,\im,
\\ K_{3,1}  &=  - 0.54527032667963220005 - 0.23807964635130112660\,\im.
\end{align*}
The choice of truncation indices ensures that all displayed digits are accurate. That these values must be complex conjugates follows directly from the fact that the $a_{s,1}$ are real.

To compute $K_{1,2}$ and $K_{3,2}$, we apply \eqref{ahyper} with $j=2$ and $\ell=1$. In particular, if
\[
N_{1,3}^{(1)}  \approx \bigg(  1 - \frac{\alpha _2^{(0)} }{\alpha _2^{(1)} } \bigg)N_2^{(0)}  = \tfrac{1}{2}N_1^{(0)} ,
\]
then
\begin{gather}\label{latecoeff2}
\begin{split}
a_{N_2^{(0)},2}& =  (1 - \im) K_{1,2} \left(\frac{\im}{2 - \im}\right)^{N_2^{(0)}}
   \sum_{s = 0}^{N_1^{(1)} - 1}
      a_{s,1} \left(\frac{2}{2 - \im}\right)^{-s}
      \Gamma\big(N_2^{(0)} - s\big) \\
& \quad- \im K_{3,2} \left(\frac{1}{2}\right)^{N_2^{(0)}}
   \sum_{s = 0}^{N_3^{(1)} - 1}
      a_{s,3}\left(\frac{1}{2}\right)^{-s}
      \Gamma \big(N_2^{(0)} - s\big)
 + \left( \frac{1}{4} \right)^{N_2^{(0)} } \Gamma \big( N_2^{(0)}  + \tfrac{3}{2}\big)\mathcal{O}(1)
\\ & = - \im K_{3,2} \left(\frac{1}{2}\right)^{N_2^{(0)}}
   \sum_{s = 0}^{N_3^{(1)} - 1}
      a_{s,3}\left(\frac{1}{2}\right)^{-s}
      \Gamma \big(N_2^{(0)} - s\big)
 + \left( \frac{1}{\sqrt{5}} \right)^{N_2^{(0)} } \Gamma \big( N_2^{(0)}\big)\mathcal{O}(1),
\end{split}
\end{gather}
as $N_2^{(0)}\to+\infty$. We first determine $ K_{3,2}$. Using the recurrence relation, we compute $a_{450,2}$ to $40$ significant figures:
\begin{gather}\label{coeffnum2}
\begin{split}
a_{450,2}=&-2.879560037976494995808298143930528338498\ldots\times10^{861} \\ &- 
3.148525898690808216545318876400345595695\ldots\times10^{861} \,\im.
\end{split}
\end{gather}
Next, neglecting the remainder, we use the second form of \eqref{latecoeff2} with $N_2^{(0)}=450$ and $N_3^{(1)}=50$, and equate it to the precomputed value \eqref{coeffnum2} of $a_{450,2}$. Solving this equation gives $K_{3,2}$ to sufficient precision to compute $K_{1,2}$. To compute $K_{1,2}$, we first calculate
\begin{equation}\label{coeffnum3}
a_{100,2}= -1.6249326747146250691 \ldots\times10^{125} - 1.8597792050637335402\ldots\times 10^{125} \,\im
\end{equation}
using the recurrence relation. Then, neglecting the remainder, we use the first form of \eqref{latecoeff2} with $N_2^{(0)}=100$, $N_1^{(1)}= N_3^{(1)}=50$, and the previously computed $K_{3,2}$, and equate the result to \eqref{coeffnum3}. Solving this yields
\[
K_{1,2} = -0.15606547412085437334 - 0.06814237098247924424 \,\im.
\]
The truncation indices are chosen so that all displayed digits are accurate. The value of $K_{3,2}$, truncated here to match the precision of $K_{1,2}$, is
\[
K_{3,2}  = 0.23345811424083407518 - 0.21637581319882182578\,\im.
\]
In an analogous manner, by considering the asymptotics of $a_{N_3^{(0)},3}$, we obtain
\begin{align*}
K_{1,3} = -0.15606547412085437334 + 0.06814237098247924424 \,\im,
\\ K_{2,3}  =\phantom{-} 0.23345811424083407518 + 0.21637581319882182578\,\im.
\end{align*}

We now have all the ingredients in place to write down the hyperasymptotic approximations. At level zero,
\[
\beta_1^{(0)}=\frac{2}{2 + \sqrt 5 },
\]
and the corresponding superasymptotic approximation is
\[
 w_{1}(z,0)=2^{-z}\sum_{s = 0}^{N_{1}^{(0)} - 1} a_{s,1} \Gamma\big(z + \tfrac{1}{2} - s\big)+R_1^{(0)}(z,0),
\]
with optimal truncation index $N_{1}^{(0)}=14$.

Proceeding to level one,
\[
\beta_1^{(0)}=\frac{2 + \sqrt 5 }{4 + \sqrt 5 },\quad
\beta_2^{(1)}=\beta_3^{(1)}=\frac{2}{4 + \sqrt 5 },
\]
and the corresponding hyperasymptotic expansion is
\begin{align*}
 w_{1}(z,0) & =2^{-z}\Bigg(
\sum_{s = 0}^{N_{1}^{(0)} - 1} a_{s,1} \Gamma\big(z + \tfrac{1}{2} - s\big)\Bigg.
 \\ &\quad + \Gamma\big(z + \tfrac{1}{2} - N_{1}^{(0)} + 1\big) K_{2,1}
\sum_{s = 0}^{N_{2}^{(1)}-1} a_{s,2}
H^{(2)}\!\left(
z;\mytop{\tfrac{1}{2} - N_{1}^{(0)},}{2}
\mytop{ N_{1}^{(0)} - s}{\im-2}
\right)\\ &\quad
+ \Gamma\big(z + \tfrac{1}{2} - N_{1}^{(0)} + 1\big) K_{3,1}
\sum_{s = 0}^{N_{3}^{(1)}-1} a_{s,3}
H^{(2)}\!\left(
z;\mytop{\tfrac{1}{2} - N_{1}^{(0)},}{2}
\mytop{ N_{1}^{(0)} - s}{-\im-2}
\right)\Bigg.\Bigg)+ R_1^{(1)}(z,0),
\end{align*}
with optimal truncation indices $N_{1}^{(0)}=20$ and $N_{2}^{(1)}=N_{3}^{(1)}=10$.

Finally, at level two, we obtain
\[
\beta_1^{(0)} = \frac{4+\sqrt{5}}{6+\sqrt{5}}, \quad 
\beta_2^{(1)} = \beta_3^{(1)} = \frac{4}{6+\sqrt{5}}, \quad \beta_1^{(2)} =  \beta_2^{(2)} = \beta_3^{(2)} = \frac{2}{6+\sqrt{5}},
\]
and the corresponding hyperasymptotic expansion is
{\allowdisplaybreaks\begin{align*}
& w_{1}(z,0)  =2^{-z}\Bigg(
\sum_{s = 0}^{N_{1}^{(0)} - 1} a_{s,1} \Gamma\big(z + \tfrac{1}{2} - s\big)\Bigg.
 \\ &\quad + \Gamma\big(z + \tfrac{1}{2} - N_{1}^{(0)} + 1\big) K_{2,1}
\sum_{s = 0}^{N_{2}^{(1)}-1} a_{s,2}
H^{(2)}\!\left(
z;\mytop{\tfrac{1}{2} - N_{1}^{(0)},}{2}
\mytop{ N_{1}^{(0)} - s}{\im-2}
\right)\\ &\quad
+ \Gamma\big(z + \tfrac{1}{2} - N_{1}^{(0)} + 1\big) K_{3,1}
\sum_{s = 0}^{N_{3}^{(1)}-1} a_{s,3}
H^{(2)}\!\left(
z;\mytop{\tfrac{1}{2} - N_{1}^{(0)},}{2}
\mytop{ N_{1}^{(0)} - s}{-\im-2}
\right)\\
 &\quad+\Gamma\big(z + \tfrac{1}{2} - N_1^{(0)} + 1 \big) K_{2,1} K_{1,2} 
   \sum_{s=0}^{N_1^{(2)} - 1} 
   a_{s,1} H^{(3)}\!\left(z;
   \mytop{\tfrac{1}{2} - N_1^{(0)},}{2}
   \mytop{N_1^{(0)} - N_2^{(1)} + 1,}{\im-2}
   \mytop{N_2^{(1)} - s}{2-\im}\right) \\
&\quad+ \Gamma\big(z + \tfrac{1}{2} - N_1^{(0)} + 1 \big) K_{2,1} K_{3,2} 
   \sum_{s=0}^{N_3^{(2)} - 1} 
   a_{s,3} H^{(3)}\!\left(z;
   \mytop{\tfrac{1}{2} - N_1^{(0)},}{2}
   \mytop{N_1^{(0)} - N_2^{(1)} + 1,}{\im-2}
   \mytop{N_2^{(1)} - s}{-2\im}\right) \\
&\quad+ \Gamma\big(z + \tfrac{1}{2} - N_1^{(0)} + 1 \big) K_{3,1} K_{1,3} 
   \sum_{s=0}^{N_1^{(2)} - 1} 
   a_{s,1} H^{(3)}\!\left(z;
   \mytop{\tfrac{1}{2} - N_1^{(0)},}{2}
   \mytop{N_1^{(0)} - N_3^{(1)} + 1,}{-\im-2}
   \mytop{N_3^{(1)} - s}{2+\im}\right) \\
&\quad+ \Gamma\big(z + \tfrac{1}{2} - N_1^{(0)} + 1 \big) K_{3,1} K_{2,3} 
   \sum_{s=0}^{N_2^{(2)} - 1} 
   a_{s,2} H^{(3)}\!\left(z;
   \mytop{\tfrac{1}{2} - N_1^{(0)},}{2}
   \mytop{N_1^{(0)} - N_3^{(1)} + 1,}{-\im-2}
   \mytop{N_3^{(1)} - s}{2\im}\right)\Bigg.\Bigg)\\
  &\quad + R_2^{(1)}(z,0),
\end{align*}}
with optimal truncation indices $N_1^{(0)}  = 23$, $N_2^{(1)}  = N_3^{(1)}  = 15$, and $N_1^{(2)} = N_2^{(2)}  = N_3^{(2)}  = 7$.

To compute the ``exact" value of $w_1(z,0)$ at $z = 30 + \mathrm{i}$, we first compute $w_1(z,0)$ at $z = 100 + \mathrm{i}$, $101 + \mathrm{i}$, and $102 + \mathrm{i}$ using 47 terms of its inverse factorial expansion, and then apply \eqref{example2} backwards to reach $z = 30 + \mathrm{i}$. The numerical results are presented in Table \ref{table1}. Figure \ref{figure1} shows the magnitudes of the terms in the level-two hyperasymptotic expansion, normalised by $2^{-z}\Gamma\left(z+\frac{1}{2}\right)$.

\begin{table}[htbp]
\centering
\caption{Hyperasymptotic approximations of $w_1(z,0)$ at $z = 30 + \im$}
\label{table1}
\begin{tabularx}{\textwidth}{@{\hspace{-1.5em}} 
>{\centering\arraybackslash}X @{\hspace{-1.5em}}
>{\centering\arraybackslash}p{6cm}@{\hspace{0.5em}}
>{\centering\arraybackslash}X @{\hspace{0.5em}}
}
\midrule
\addlinespace[1.2ex]
level &  approximation & $|$relative error$|$ \\
\addlinespace[1ex]
\midrule
\addlinespace[0.5em]
0 & \parbox[c]{6cm}{
$-4.8415547386561052316679\times 10^{22}$\\
$\phantom{00} +2.2760201589599892230494\times 10^{22} \im$} & $6.8 \times 10^{-11}$ \\[3ex]
1 & \parbox[c]{6cm}{
$-4.8415547384705057894803\times 10^{22}$\\
$\phantom{00} +2.2760201586435914394371\times 10^{22} \im$} & $8.5 \times 10^{-17}$ \\[3ex]
2 & \parbox[c]{6cm}{
$-4.8415547384705057443385\times 10^{22}$\\
$\phantom{00} +2.2760201586435909838540\times 10^{22} \im$} & $8.6 \times 10^{-22}$ \\[3ex]
exact & \parbox[c]{6cm}{
$-4.8415547384705057443395\times 10^{22}$\\
$\phantom{00}+2.2760201586435909838494\times 10^{22} \im$} & 0 \\[2ex]
\midrule
\end{tabularx}
\end{table}

\begin{figure}[htbp]
    \centering
    \includegraphics[width=0.8\textwidth]{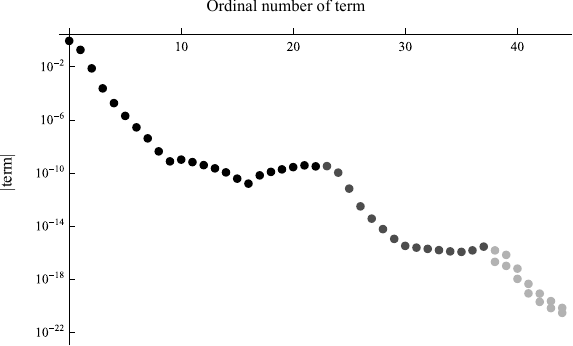}
    \caption{Magnitudes of the level-two hyperasymptotic terms of $w_1(z,0)$ at $z = 30 + \im$, normalised by scaling with $2^{-z}\Gamma \big(z+\tfrac{1}{2}\big)$.}
    \label{figure1}
\end{figure}

\section{Discussion}\label{sec7}

We have extended the hyperasymptotic method to inverse factorial series solutions of higher-order linear difference equations with polynomial coefficients and an irregular singularity of rank one at infinity. The resulting expansions were shown to determine the solutions uniquely, and we have also indicated how the associated connection coefficients can be computed numerically using hyperasymptotic techniques. In addition, we have derived explicit remainder bounds for the inverse factorial series solutions. 

The analysis relied on Mellin--Borel transforms, whose Laplace-type transforms establish a direct link with the hyperasymptotic theory of Laplace-type integrals developed in \cite{OldeDaalhuis1998b}.

The assumption that the coefficients $f_k(z)$ of the difference equation are polynomials ensures that the Mellin--Borel transforms possess only algebraic-type singularities, located at the roots of the characteristic equation associated with the difference equation. If the $f_k(z)$ are instead allowed to have (isolated) singularities in the finite plane, these give rise to additional singularities of the Mellin--Borel transforms at the origin. In particular, if the singularities are poles, then---as suggested in \cite{OldeDaalhuis2004}---the Mellin--Borel transforms develop singularities at the origin that are described by powers of logarithms, and they satisfy linear differential equations of order higher than that of the original difference equation. It is unclear whether an exact expression for the logarithmic--power singularity exists in terms of the Mellin--Borel transforms of the other solutions, analogous to \eqref{connection} at the algebraic singularities. Such an expression would be essential for developing hyperasymptotic refinements of inverse factorial series solutions in this more general setting.

Another possible extension concerns difference equations with an irregular singularity of higher integer rank at infinity. In the case of polynomial potentials, one expects a connection with the hyperasymptotic theory for linear differential equations possessing a higher-rank irregular singularity at infinity. This hyperasymptotic theory was investigated in the paper \cite{Murphy1997} for second-order equations and in the PhD thesis \cite{Murphy2001} for higher-order equations. However, these works do not employ the Borel transform, instead deriving their results via the Cauchy--Heine method.

Finally, it would be interesting to study the case where the characteristic equation \eqref{chareq} has multiple roots. Asymptotic expansions for solutions in this setting have been constructed for second-order difference equations, for example in \cite{Culmer1963} and \cite{Wong1992}, while the higher-order case was studied in \cite{Adams1928}. In these works, the asymptotic expansions are expressed in descending powers of $z$. For second-order equations, the form of these expansions suggests that an analogue of the inverse factorial series solutions could take the form
\[
\Gamma ^2 (z)\lambda ^{ - z} \sum_{s = 0}^\infty  a_s U(z,\mu  - s,\nu ),
\]
where $U$ is the confluent hypergeometric function, $\lambda$ is the double root of the characteristic equation, and the constants $\mu$, $\nu$, and the coefficients $a_s$ depend on $\lambda$ as well as on the coefficients of the difference equation. A second, linearly independent solution might likewise be expressed in terms of the confluent hypergeometric function $M$. Determining an appropriate analogue of the Mellin--Borel transform for this case requires further investigation.

\section*{Acknowledgement}
The author gratefully acknowledges the referee for his/her insightful comments and valuable suggestions, which have significantly improved the clarity and presentation
of the paper.

\appendix
\section{Computation of the hyperterminants}\label{Appendix}

In this appendix, we present an efficient method for computing the hyperterminants $H^{(\ell + 1)}$.

\begin{proposition}
Let $\ell \ge 1$, and suppose that $\Re(z + M_0 + 1) > 0$, $\Re(M_2) > 2$, and $\Re(M_r) > 1$ for all $r = 1, 3, \ldots, \ell$. Under these conditions, the following convergent expansion holds:
\begin{align*}
& H^{(\ell + 1)}\!\left(z;\mytop{M_0,}{\sigma_0,}
\mytop{\ldots,}{\ldots,}
\mytop{M_\ell}{\sigma_\ell}\right)=\left( \frac{\e^{\pi \im}}{\sigma_0} \right)^{M_0 + 1}
\left( \e^{-\pi \im} \sum_{r = 0}^\ell \sigma_r \right)^{1-\ell+\sum_{r = 0}^\ell M_r}
\sum_{p = 0}^\infty  \left( -\frac{\sigma _0 }{\sigma _1 } \right)^{p + 1} \\ & \hspace{5em}
\times A^{(\ell )}\!\left(p; \mytop{M_1,}{\sigma_1,}
\mytop{\ldots,}{\ldots,}
\mytop{M_\ell}{\sigma_\ell}\right) \frac{1}{\left( z + M_0  + M_1  \right)_{p + 1} } {}_2 F_1\! \left( 
   \mytop{p + 1,M_1  + p}{z + M_0  + M_1  + p + 1}  ;1 + \frac{\sigma _0 }{\sigma _1 } \right),
\end{align*}
where
\[
A^{(1)}\!\left(
  p; \mytop{M_0}{\sigma_0}
\right)
= \delta _{p,0} \e^{\pi \im M_0 } \sigma _0^{1 - M_0 } \Gamma \left( M_0 \right),
\]
with $\delta_{p,0}$ denoting the Kronecker delta, and for $\ell \ge 1$,
\begin{align*}
& A^{(\ell + 1)}\!\left(p;\mytop{M_0,}{\sigma_0,}
\mytop{\ldots,}{\ldots,}
\mytop{M_\ell}{\sigma_\ell}\right) = \e^{\pi \im M_0} \sigma _0^{\ell  - M_0 } \Gamma \left( M_0  + p\right)\sum_{s = 0}^\infty  \left( - \frac{\sigma _0 }{\sigma _1 } \right)^{s + 1} \\ & \hspace{5em}
\times A^{(\ell)}\!\left(s;\mytop{M_1,}{\sigma_1,}
\mytop{\ldots,}{\ldots,}
\mytop{M_\ell}{\sigma_\ell}\right)\frac{\left( s + 1 \right)_p }{\left( M_0  + M_1  - 1 \right)_{p + s + 1} } {}_2F_1\! \left( \mytop{p + s + 1,M_1  + s}{M_0  + M_1  + p + s} ;1 + \frac{\sigma _0 }{\sigma _1 } \right).
\end{align*}
\end{proposition}

The Gauss hypergeometric functions ${}_2F_1$ appearing in the proposition can be evaluated using standard numerical techniques; see, for example, \cite[\href{https://dlmf.nist.gov/15.19}{\S15.19}]{DLMF} and the references therein for computational methods and implementations.

\begin{proof} Let $\ell \ge 1$, and suppose that $\Re(z + M_0 + 1) > 0$ and $\Re(M_r) > 1$ for $r = 1, \ldots, \ell$. From the definition \eqref{hyperF} of the hyperterminants, it follows directly that
\[
F^{(\ell  + 1)} \! \left( 0;\mytop{z+M_0+2,}{\sigma_0,}
\mytop{M_1,}{\sigma_1,}
\mytop{\ldots,}{\ldots,}
\mytop{M_\ell}{\sigma_\ell} \right) 
= - \int_0^{[\pi - \arg \sigma_0]} 
\e^{\sigma_0 t} t^{z + M_0} 
F^{(\ell)}\!\left(t;\mytop{M_1,}{\sigma_1,}
\mytop{\ldots,}{\ldots,}
\mytop{M_\ell}{\sigma_\ell}\right) \id t.
\]
If we additionally assume that $\Re(M_2) > 2$, then the function $F^{(\ell)}$ admits the absolutely and uniformly convergent expansion
\[
F^{(\ell)}\!\left(t;\mytop{M_1,}{\sigma_1,}
\mytop{\ldots,}{\ldots,}
\mytop{M_\ell}{\sigma_\ell}\right) = \sum_{p = 0}^\infty  A^{(\ell)}\!\left(p; \mytop{M_1,}{\sigma_1,}
\mytop{\ldots,}{\ldots,}
\mytop{M_\ell}{\sigma_\ell}\right) U\left( p + 1,2 - M_1 ,\sigma _1 t \right),
\]
where $|\arg(\sigma_1 t)| \le \pi$, and $U$ denotes the confluent hypergeometric function (see \cite[Theorems 2 and 3]{OldeDaalhuis1998a}). Substituting this into the previous expression yields
\begin{align*}
F^{(\ell  + 1)} \! \left( 0;\mytop{z+M_0+2,}{\sigma_0,}
\mytop{M_1,}{\sigma_1,}
\mytop{\ldots,}{\ldots,}
\mytop{M_\ell}{\sigma_\ell} \right)   &=  - \sum_{p = 0}^\infty A^{(\ell)}\!\left(p; \mytop{M_1,}{\sigma_1,}
\mytop{\ldots,}{\ldots,}
\mytop{M_\ell}{\sigma_\ell}\right)\\ & \quad \times \int_0^{\left[ {\pi  - \arg \sigma _0 } \right]} \e^{\sigma _0 t} t^{z + M_0 } U\left( p + 1,2 - M_1 ,\sigma _1 t \right)\id t.
\end{align*}
Assume temporarily that $\frac{\pi}{2} < \arg \sigma _0  - \arg \sigma _1  < \frac{3\pi}{2}$. Using the known Laplace transform of the confluent hypergeometric function $U$ \cite[Eq.~ \href{http://dlmf.nist.gov/13.10.E7}{(13.10.7)}]{DLMF} together with a linear transformation formula for ${}_2F_1$ \cite[Eq.~\href{http://dlmf.nist.gov/15.8.E1}{(15.8.1)}]{DLMF}, we obtain
\begin{align*}
&\int_0^{[\pi - \arg \sigma_0]} \e^{\sigma_0 t} t^{z + M_0} 
  U\left(p + 1, 2 - M_1, \sigma_1 t \right) \id t \\
&= \left( \frac{\e^{\pi \im}}{\sigma_0} \right)^{z + M_0 + 1} 
  \frac{\Gamma\left(z + M_0 + 1\right)}{\left(z + M_0 + M_1\right)_{p + 1}} 
  {}_2F_1 \!\left( 
      \mytop{p + 1, z + M_0 + 1}{z + M_0 + M_1 + p + 1} ; 1 + \frac{\sigma_1}{\sigma_0} 
  \right) \\
&= \left( \frac{\e^{\pi \im}}{\sigma_0} \right)^{z + M_0 + 1} 
  \frac{\Gamma\left(z + M_0 + 1\right)}{\left(z + M_0 + M_1\right)_{p + 1}} 
  \left( -\frac{\sigma_0}{\sigma_1} \right)^{p + 1} 
  {}_2F_1\! \left( 
    \mytop{p + 1, M_1 + p }{z + M_0 + M_1 + p + 1} ; 1 + \frac{\sigma_0}{\sigma_1} 
  \right).
\end{align*}
Accordingly, we find that
\begin{align*}
& F^{(\ell  + 1)} \! \left( 0;\mytop{z+M_0+2,}{\sigma_0,}
\mytop{M_1,}{\sigma_1,}
\mytop{\ldots,}{\ldots,}
\mytop{M_\ell}{\sigma_\ell} \right)  =- \left( \frac{\e^{\pi \im} }{\sigma _0 }\right)^{z + M_0  + 1} \Gamma \left( z + M_0  + 1 \right)\sum_{p = 0}^\infty \left(  - \frac{\sigma _0 }{\sigma _1 } \right)^{p + 1}  \\ & \hspace{5em} \times A^{(\ell)}\!\left(p; \mytop{M_1,}{\sigma_1,}
\mytop{\ldots,}{\ldots,}
\mytop{M_\ell}{\sigma_\ell}\right)\frac{1}{\left( z + M_0  + M_1 \right)_{p + 1}}{}_2F_1\! \left( 
    \mytop{p + 1, M_1 + p }{z + M_0 + M_1 + p + 1} ; 1 + \frac{\sigma_0}{\sigma_1} 
  \right).
\end{align*}
The constraint on $\arg \sigma _0  - \arg \sigma _1$ may now be removed through analytic continuation. Substituting this expression into \eqref{Hdef} yields the desired expansion for $H^{(\ell+1)}$. The recursive formula for the coefficients $A^{(\ell+1)}$ follows from \cite[Theorem 3]{OldeDaalhuis1998a}, using a linear transformation of the Gauss hypergeometric function  \cite[Eq.~\href{http://dlmf.nist.gov/15.8.E1}{(15.8.1)}]{DLMF}.
\end{proof}

\let\oldbibitem\bibitem
\def\bibitem{\par\addvspace{3pt}\oldbibitem}

\vspace{-6.3pt}

\end{document}